\documentclass[11pt]{article}

\usepackage{ifpdf}
\ifpdf 
    \usepackage[pdftex]{graphicx}   
    \usepackage[pdftex,     
            plainpages=false,   
            breaklinks=true,    
            colorlinks=true,
            pdftitle=My Document
            pdfauthor=My Good Self
           ]{hyperref} 
    \usepackage{mathtools,xparse}
      \usepackage{amsthm,amsmath,amssymb}
\usepackage{mathrsfs}
\usepackage{bm}
    \usepackage{subfigure}
      \usepackage{float}
    \newtheorem{theorem}{Theorem}[section]
\newtheorem{lemma}{Lemma}[section]
    \newtheorem{remark}{Remark}[section]
    \usepackage{mathrsfs}
    \usepackage{indentfirst}
\usepackage{titlesec}
\usepackage{setspace}
\usepackage{geometry}
\usepackage{scalerel,stackengine}

\stackMath
\newcommand\widecheck[1]{%
\savestack{\tmpbox}{\stretchto{%
  \scaleto{%
    \scalerel*[\widthof{\ensuremath{#1}}]{\kern-.6pt\bigwedge\kern-.6pt}%
    {\rule[-\textheight/2]{1ex}{\textheight}}
  }{\textheight}%
}{0.5ex}}%
\stackon[1pt]{#1}{\scalebox{-1}{\tmpbox}}%
}
\newcommand{\RNum}[1]{\uppercase\expandafter{\romannumeral #1\relax}}
\DeclarePairedDelimiter{\norm}{\lVert}{\rVert}
\NewDocumentCommand{\normL}{ s O{} m }{%
  \IfBooleanTF{#1}{\norm*{#3}}{\norm[#2]{#3}}_{L_2(\Omega)}%
}
\else 
    \usepackage{graphicx}       
    \usepackage{hyperref}       
\fi 
\newcommand{\vect}[1]{\bf #1}

\usepackage{enumitem}

\graphicspath{{./figures/}}  

\title{Multi-symplectic discontinuous Galerkin methods for the stochastic Maxwell equations with additive noise}
\author{	Jiawei Sun\footnote{Department of Mathematics, The Ohio State University,
		Columbus, OH 43210, USA. E-mail: sun.2261@buckeyemail.osu.edu.} \and
		Chi-Wang Shu\footnote{Division of Applied Mathematics, Brown University, Providence, RI 02912, USA. 
E-Mail: chi-wang\_shu@brown.edu. The work of this author is partially supported by NSF grant DMS-2010107
and AFOSR grant FA9550-20-1-0055.} 
		\and Yulong Xing\footnote{Department of Mathematics, The Ohio State University,
		Columbus, OH 43210, USA. E-mail: xing.205@osu.edu. The work of this author is partially supported by the NSF grant DMS-1753581.}}
\date{}

\begin{document}
\maketitle
\begin{abstract} 
One- and multi-dimensional stochastic Maxwell equations with additive noise are considered in this paper. 
It is known that such system can be written in the multi-symplectic structure, and the stochastic energy increases linearly in time. 
High order discontinuous Galerkin methods are designed for the stochastic 
Maxwell equations with additive noise, and we show that the proposed methods satisfy the discrete form of the 
stochastic energy linear growth property and preserve the multi-symplectic structure on the discrete level.
 Optimal error estimate of the semi-discrete DG method is also analyzed. The fully discrete methods are obtained 
by coupling with symplectic temporal discretizations. One- and two-dimensional numerical results are provided to demonstrate the performance of the proposed methods, and optimal error estimates and linear growth of the discrete energy can be observed for all cases.  
\end{abstract}


\section{Introduction}
\setcounter{equation}{0} \setcounter{figure}{0}\setcounter{table}{0}

In this paper we develop and analyze high order discontinuous Galerkin (DG) methods for one- and two-dimensional 
stochastic Maxwell equations with additive noise. Maxwell equations play an important role in many physical applications, 
and have been widely used in electromagnetism, electronic biology, optical imaging, etc. The general formulation of 
Maxwell equations is 
\begin{equation}\label{Max}
  \begin{cases}   
     \partial_t\vect{D}=\nabla\times\vect{H}-{\vect{J}}_e,~~ \nabla\cdot\vect{D}=\rho,\\
    \partial_t\vect{B}=-\nabla\times\vect{E},~~\nabla\cdot  \vect{B} = 0,
  \end{cases}
\end{equation}
where $\bf{H}$ represents the magnetic field, $\bf{E}$ stands for the electric 
field, $\bf{D}$ and $\bf{B}$ are the electric and magnetic flux density 
respectively.  $\bf{J}_e$ is the electric current density, and $\rho$ is the 
electric charge density. 

Stochastic Maxwell equations are the generalized version of the deterministic Maxwell equations, which are often described as a random perturbation of the electric current density or the magnet current density by noise. The noises are commonly regarded as Brownian motion, Poisson process, etc. In \cite{Rytov}, Rytov et al. introduced fluctuations of an electromagnetic field to obtain stochastic  Maxwell equations. Ord showed in \cite{ORD} that the random walk model due to Mark Kac can be modified to produce Maxwell's field equations in 1+1 dimensions. In \cite{BL2010},
 Liaskos et al. studied the stochastic integrodifferential equations in Hilbert spaces, and they examined the well posedness for the 
 Cauchy problem of the integrodifferential equations describing Maxwell equations.  The random electromagnetic fields using the spectral representation 
 is explored in \cite{GL 2015}, and the electromagnetic fields were coupled by Maxwell equations with a random source term.
  Finite element approximations of a class of nonlinear stochastic wave equations with multiplicative noise were recently investigated in \cite{LWX2021}. 
  The semilinear stochastic Maxwell equations with additive noise in the following form: 
\begin{align}\label{MUL}
       \begin{cases}
   \epsilon d\vect{E} -\nabla\times 
    \vect{H}dt=-\vect{J}_e(t,\vect{x},\vect{E},\vect{H})dt-\vect{J}_e^r(t,\vect{x})\circ dW, \\
    \mu d\vect{H}+\nabla\times\vect{E}dt=-\vect{J}_m(t,\vect{x},\vect{E},\vect{H})dt-\vect{J}_m^r(t,\vect{x})\circ 
    dW,
  \end{cases}
\end{align}
were studied by Chen et al. in \cite{CC 2016}, where $dW$ is a space-time mixed color noise, often driven by Brownian motion, and $\vect{J}_e$ and $\vect{J}_m$ are described as electric current and magnetic current, along with continuous bounded functions. Theoretical properties of the stochastic system \eqref{MUL} such as regularity, energy and divergence evolution law, and symplecticity have been presented in that paper. In addition, a stochastic Runge-Kutta semidiscretization scheme was proposed for \eqref{MUL} and proven to possess first order of mean accuracy. 
In \cite{CCHS}, David et al. developed an exponential integrator for a more generalized formulation of \eqref{MUL}, when $\bf{J}_e^r$ and $\bf{J}_m^r$ depend on $\bm{E}$ and $\bm{H}$. In a recent review article, Zhang et al. \cite{LZ2019} presented different types of stochastic Maxwell equations with additive or multiplicative noises. 

Stochastic Maxwell equations can be viewed as a type of stochastic Hamiltonian PDEs. In \cite{SJ2012}, Jiang et al. considered stochastic Hamiltonian PDEs in the form
\begin{equation}\label{defSPDE}
  Mdz+Kz_xdt=\nabla_zS_1(z)dt+\nabla_zS_2(z)dW_t,
\end{equation}
where $M$ and $K$ are anti-symmetric matrices, and $S_1$ and $S_2$ are smooth functions of $z$.  
It can be shown that the system \eqref{defSPDE} satisfies the following stochastic multi-symplectic conservation law:
\[d\omega+\partial_x \kappa dt=0,\qquad \omega=MU\cdot V,\qquad \kappa=KU\cdot V,\]
which $U$ and $V$ are a pair of solutions to the variational equation
\[Md(\partial z)+K(\partial z)_x dt=\nabla^2_zS_1(z)\partial zdt+\nabla^2_zS_2(z)\partial zdW_t.\]
In \cite{CHZ2016,HJZ2014}, multi-symplectic finite difference methods have been studied for the stochastic Maxwell equations with additive noise of the form 
 \begin{align}\label{ADD}
     \begin{cases}
          \epsilon d\vect{E} =\nabla \times \vect{H}dt- \lambda_1\vect{1}^T dW,\\
     \mu d\vect{H}=-\nabla\times\vect{E}+\lambda_2\vect{1}^TdW,
     \end{cases}
 \end{align}
which can be reformulated in two slightly different formulations of \eqref{defSPDE} in \cite{CHZ2016,HJZ2014} with different set of auxiliary variables introduced. Numerical methods based on these reformation have been presented and studied, and it was shown that these methods preserve stochastic multi-symplecticity on the discrete level. In addition, the linear growth property of stochastic energy was also preserved by the proposed methods. In \cite{HJ2017}, the extension to stochastic Maxwell equations with multiplicative noise was investigated.

In this work, we investigate the high-order schemes for stochastic Maxwell equations, following the recent work in \cite{YL 2020} on DG methods for stochastic conservation laws. The DG method is a class of finite element methods that uses discontinuous piecewise polynomials as the basis functions. This method has been shown to adopt many advantages from both finite element and finite volume methods, which includes hp-adaptivity flexibility, efficient parallel implementation, the ability of handling complicated boundary conditions, etc. The DG methods were first introduced in \cite{RH1973} by Reed and Hill to solve transport equations, and later they were extended to solve hyperbolic conservation laws by Cockburn et al. in \cite{CHS1990,  CKS2000,  CLS1989, CS1989}. 
There have been some recent studies in extending DG method for stochastic partial differential equations. In \cite{YL 2020} Li et al. applied the DG method to the stochastic conservation laws with multiplicative noise
 \[du+f(u)_xdt=g(x,t,u)dW_t. \]
In \cite{YuL 2020}, they also proposed an ultra-weak DG method for the stochastic Korteweg-De Vries equations in the form 
 \[du=-(u_{xxx}+f(u)_x)dt+g(x,t,u)dW_t.\]
Optimal error estimate was proven for the semilinear equations, and numerically optimal convergence rate was also observed for many nonlinear cases. 

In this paper we apply high order DG methods to stochastic Maxwell equations with additive noise \eqref{ADD} in 
one and two dimensions. The stochastic energy of the exact solutions are shown to satisfy the linear growth property, and we will demonstrate that the numerical solutions of the semi-discrete DG methods satisfy the similar energy law on the discrete level. When the standard Brownian motion $W_t$ is considered, the exactly same discrete energy law can be obtained. Following the error estimate for the deterministic Maxwell equations, we provide the optimal error estimate of the semi-discrete DG methods for the one- and two-dimensional stochastic Maxwell equations on cartesian meshes.  Furthermore, multi-symplectic property of certain DG methods for the one-dimensional deterministic multi-symplectic Hamiltonian partial differential equations (HPDEs) of the form
 \[Mz_t+Kz_x=\nabla_z S(z)\]  
has been recently investigated in \cite{SX2020}. 
Following the idea, we will start by establishing the multi-symplectic structure of the stochastic Maxwell equations, and then prove that the DG methods with suitable numerical fluxes can preserve the multi-symplecticity. The resulting semi-discrete methods are combined with symplectic temporal discretization. Both first order symplectic Euler method and second order symplectic partitioned Runge-Kutta (PRK) method will be introduced and analyzed in this paper. 

The structure of this paper is as follows.  Section \ref{msltc1d} provides 
theoretical results for one-dimensional stochastic Maxwell equations. We discuss the 
conservation of multi-symplecticity of our DG scheme, demonstrate the energy 
law of numerical solutions, and present the result on the optimal error estimate of 
the proposed methods.  Section \ref{2dmtsplctc} provides the same theoretical 
results on DG methods for two-dimensional stochastic Maxwell equations. Temporal  
discretization is provided in section \ref{time}. We consider both symplectic 
Euler method and the second order symplectic PRK method. Numerical 
results in both one and two dimensions are provided in section \ref{numerical} to validate the convergence rate and the linear energy growth property. Section \ref{conclusion} contains some conclusion remarks.

Throughout this paper, $L^2$ norm is denoted by $\norm{\cdot}$, and $C$ represents a generic 
positive constant independent of the spatial and temporal step size $h$ and $\Delta t$, 
which can take different values in different cases. In general, $W_t$ represents a 
Q-Wiener process defined on a given probability space 
$(\Omega,\mathcal{F},\mathbb{P})$, which can be characterized as 
\begin{equation}\label{defW}
  W_t=W(t,\vect{x},\omega)=\sum^\infty_{m=1}\sqrt{\gamma_m}e_m(\vect{x})\mathcal{B}_m(t,\omega), 
  t\ge 0, 
  \vect{x}\in \mathbb{R}~ \text{or}~ \mathbb{R}^2,
\end{equation}
where $\{e_m\}$ is an orthonormal basis of $L^2(\mathbb{D})$, with $\mathbb{D}\subset \mathbb{R}^d,~d=1,2$. $Q$ is a symmetric, non-negative operator such that $Tr(Q)<\infty$ and $Qe_m=\gamma_m e_m$. Furthermore, $\{\mathcal{B}_m\}$ is a sequence of independent standard Brownian motions.

 \section{One-dimensional Maxwell equation with additive noise}\label{msltc1d}
 \setcounter{equation}{0} \setcounter{figure}{0}\setcounter{table}{0}
 
In this section, the one-dimensional computational domain is denoted by $I$, which is partitioned into subintervals $I_j=[x_{j-\frac12},x_{j+\frac12}]$ 
where $j=1,2,\cdots,N$. We also denote $x_j=\frac{1}{2}(x_{j-\frac12}+x_{j+\frac12})$ to be the center of each cell, and $h_j=x_{j+\frac12}-x_{j-\frac12}$ to be the mesh size. Let $h=\max_j h_j$ be the maximum mesh size. We further assume that $h/h_j$ 
is bounded over all $j$.  Assume $P^k(I_j)$ to be the space of polynomials of degree up to $k$ on $I_j$, and the piecewise polynomial space $V_h^k$ is defined as follows:
\[V_h^k=\{v:v|_{I_j}\in P^k(I_j),~~ j=1,2,\cdots,N\}.\]
Note that functions in $V_h^k$ can be discontinuous at cell interfaces. Let $v_{j+\frac12}^+$ and $v_{j+\frac12}^-$ be the right and left limit of $v$ at the interface $x_{j+\frac12}$, and we denote by $\{ v\}_{j+\frac12} = \frac12(v_{j+\frac12}^+ + v_{j+\frac12}^-)$ and $[v]_{j+\frac12} = v_{j+\frac12}^+ - v_{j+\frac12}^-$ the average and jump of $v$ at $x_{j+\frac12}$.  

We start by considering the one-dimensional (1D) stochastic Maxwell equation with additive noise of the form
  \begin{equation}\label{1}
  \begin{cases}
d\eta=-u_xdt-\lambda_1dW_t, &\\
du=-\eta_xdt+\lambda_2dW_t,
\end{cases}
\end{equation}
which can be viewed as a 1D version of the model \eqref{ADD}.
The DG scheme for \eqref{1} is formulated as:  for $x\in I,~ (\omega,t)\in \Omega\times [0,T]$, find $\eta_h(\omega,x,t), u_h(\omega,x,t) \in V^k_h$, such that for any test functions $ \varphi, \widetilde{\varphi}\in V_h^k$, it holds that 
 \begin{eqnarray}
  \int_{I_j}d\eta_h\varphi(x)dx=\Big(\int_{I_j}u_h\varphi_xdx 
  -(\widecheck{u}_h\varphi^-)_{j+\frac12}+(\widecheck{u}_h\varphi^+)_{j-\frac12}\Big)dt
  -\int_{I_j}\lambda_1\varphi dW_tdx,   \label{phi}\\ \, \notag \\
    \int_{I_j}du_h\widetilde{\varphi}(x)dx=\Big(\int_{I_j}\eta_h\widetilde{\varphi}_xdx 
  -(\widecheck{\eta}_h\widetilde{\varphi}^-)_{j+\frac12}+(\widecheck{\eta}_h\widetilde{\varphi}^+)_{j-\frac12}\Big)dt
  +\int_{I_j}\lambda_2\widetilde{\varphi} dW_tdx ,
  \label{phii}
  \end{eqnarray}
where the generalized alternating numerical fluxes are chosen to be  
$$\widecheck{u}_h=\{u_h\}+\alpha[u_h], \qquad \widecheck{\eta}_h=\{\eta_h\}-\alpha[\eta_h],$$ 
for some non-zero constant $\alpha\in[-1,1]$. 
Below, we will explore some theoretical properties of this DG method, including the discrete energy law, optimal error estimate and the multi-symplectic structure.

\subsection{Discrete energy law}
The exact solutions of the three dimensional stochastic Maxwell equation \eqref{ADD} satisfy the linear energy growth property, as studied in \cite{HJZ2014}. Below, we start by presenting similar result for the one-dimensional model \eqref{1}.
%

\begin{theorem}[\bf{Continuous energy law}]\label{1denergypdelevel}
  Let $u$ and $\eta$ be the solutions to the model \eqref{1} under periodic boundary condition. Denote the energy by $\mathcal{E}(t)=\int_I u^2(x,t)+\eta^2(x,t)dx$, then for any $t$, the global stochastic energy satisfies the following energy law
\begin{eqnarray}\label{thm2.9}
     \mathcal{E}(t) = \mathcal{E}(0)  +2 \int_0^t\int_{I}(\lambda_2\eta-\lambda_1 u) dW_tdx
+(\lambda_1^2+\lambda_2^2)Tr(Q) t.
  \end{eqnarray}
and, after taking the expectation,
  \begin{eqnarray}\label{thm2.10}
    \mathbb{E}\big(\mathcal{E}(t)\big)=\mathcal{E}(0)+(\lambda_1^2+\lambda_2^2)Tr(Q)t.
  \end{eqnarray}
\end{theorem}
\begin{proof}
  By utilizing the It\^{o}'s lemma and the equations \eqref{1}, we have 
  \begin{align}
          &d\mathcal{E}(t)=\int_I \big(2udu+ (du)^2+2\eta d\eta +(d\eta)^2\big) dx \notag\\
    &~~~~~~~=-\int_I (2u\eta_x+2\eta u_x)dtdx+2\int_I (\lambda_2u-\lambda_1 \eta) dW_tdx+(\lambda_1^2+\lambda_2^2)\int_I (dW_t)^2 dx.	\label{1denergyeq1}
  \end{align}
It follows from the periodic boundary condition and the definition of $W_t$ that
  \[\int_I (2u\eta_x+2\eta u_x)dtdx=0, \qquad\qquad \int_I (dW_t)^2 dx=Tr(Q)dt.\]
  Therefore, \eqref{1denergyeq1} reduces to
  \begin{eqnarray}\label{beforeexp}
    d\mathcal{E}(t)=2\int_I (\lambda_2u-\lambda_1 \eta)dW_tdx+(\lambda_1^2+\lambda_2^2)Tr(Q)dt.
  \end{eqnarray}
  Integrating \eqref{beforeexp} over $t$ yields \eqref{thm2.9}, and taking an expectation leads to \eqref{thm2.10}.
\end{proof}

Next, we show that the following discrete energy law is satisfied by the numerical solutions of the DG methods.
\begin{theorem}[\bf{Discrete energy law}]\label{1denergy}
Let $u_h(\omega,x,t)$ and $\eta_h(\omega,x,t)$ be the numerical solutions to the DG methods \eqref{phi} and \eqref{phii}. \\
(a): For any $t\in[0,T]$, the numerical solutions satisfy the discrete energy law
  \begin{align}
  &\norm{u_h(\omega,x,t)}^2+\norm{\eta_h(\omega,x,t)}^2 \notag\\
  &\hskip2cm
  =2\int^t_0\int_I(\lambda_2u_h-\lambda_1\eta_h)dW_sdx
+\norm{u_h(x,0)}^2+\norm{\eta_h(x,0)}^2+(\lambda_1^2+\lambda_2^2)Kt. \label{1denergy11}
\end{align}
and
  \begin{equation}
  \mathbb{E}\Big(\norm{u_h(\omega,x,t)}^2+\norm{\eta_h(\omega,x,t)}^2\Big)=\norm{u_h(x,0)}^2+\norm{\eta_h(x,0)}^2+(\lambda_1^2+\lambda_2^2)Kt, \label{1denergy11_1}
  \end{equation}
with 
  \begin{equation}
  K=\sum_{j=1}^N\sum_{l=0}^k \mu_j^l\sum^\infty_{m=1} \Big(\int_{I_j}\phi_j^l\sqrt{\gamma_m}e_mdx\Big)^2, \label{added2}
  \end{equation}
where  $\{\phi^l_j, j=0,\cdots,k\}$ represents the set of Legendre basis over cell $I_j$, and $\mu_j^l = (\int_{I_j}(\phi_j^l)^2dx)^{-1}$. \\
(b) The constant $K$ is bounded by $(k+1)Tr(Q)$ with $k$ being the polynomial degree of the DG methods. 
Moreover, if there exists some constant $\alpha>0$ such that the series $\sum^\infty_{m=1}\gamma_m (K_m)^\alpha<\infty$ with $K_m=\norm{(e_m)_x}_{{\infty}}$, 
we can show that $K=Tr(Q)+O(h^\alpha)$, i.e., $K$ is an approximation of $Tr(Q)$ appearing in the continuous energy law \eqref{thm2.9}-\eqref{thm2.10}.
  
\end{theorem}
\begin{proof}
(a): From It\^{o}'s formula, we have
 \begin{equation}\label{Ito}
    d(u_h)^2=2u_hdu_h+(du_h)^2,\qquad d(\eta_h)^2=2\eta_hd\eta_h+(d\eta_h)^2.
 \end{equation}
By taking the test functions $\varphi=\eta_h$ and $\widetilde{\varphi}=u_h$ in the DG methods \eqref{phi}-\eqref{phii}, and summing the resulting equations up, we have 
  \begin{eqnarray}\label{qqq1}
     \int_{I_j}((du_h)u_h+(d\eta_h)\eta_h) dx&=&\int_{I_j}(\lambda_2u_h-\lambda_1\eta_h)dW_tdx+\Big(\int_{I_j}u_h(\eta_h)_xdx
+\eta_h(u_h)_xdx\Big)dt \notag \\
&&+\Big( -\big((\{u_h\}+\alpha[u_h])\eta_h^-\big)_{j+\frac12} + \big((\{u_h\}+\alpha[u_h])\eta_h^+\big)_{j-\frac12}\Big)dt
\notag \\
&&+\Big( -\big((\{\eta_h\}-\alpha[\eta_h])u_h^-\big)_{j+\frac12} + \big((\{\eta_h\}-\alpha[\eta_h])u_h^+\big)_{j-\frac12}\Big)dt\notag \\
 &=&-\Theta_{j+\frac12}+\Theta_{j-\frac12} +\int_{I_j}(\lambda_2u_h-\lambda_1\eta_h)dW_tdx,
  \end{eqnarray}
  where 
  \[\Theta=\left(\frac12+\alpha\right)\eta^-_hu_h^++\left(\frac12-\alpha\right)u_h^-\eta_h^+.\]

Let us represent the numerical solutions $u_h$ in the cell $I_j$ as
\[u_h(\omega,x,t)=\sum^k_{l=0}u_j^l(\omega,t)\phi_j^l(x),\]
which leads to
\[du_h=\sum^k_{l=0}du_j^l\phi_j^l.\]
Taking the test function $\widetilde{\varphi}=\phi_j^m$, $m=0,\cdots,k$, in \eqref{phii}, we obtain 
\begin{equation}\label{system}
  \sum^k_{l=0}\Big(\int_{I_j}\phi_j^m\phi_j^ldx\Big)du_j^l=
\mathcal{A}_j(\eta_h;\phi_j^m)dt+\int_{I_j}\lambda_2\phi_j^mdW_tdx
\end{equation}
where the operator 
\[\mathcal{A}_j(f;g)=\int_{I_j}fg_xdx-\big(\widecheck{f}g^-\big)_{j+\frac12}+
\big(\widecheck{f}g^+\big)_{j-\frac12},\]
is introduced for ease of presentation.
Denote the mass matrix by $L_j$ with the $(m,l)$ entry being
$\int_{I_j}\phi_j^l\phi_j^mdx$. Note that the orthogonal Legendre basis are chosen, therefore $L_j$ is a diagonal matrix,
so is the inverse matrix $L_j^{-1}$. Let us denote $L_j^{-1}$=diag$(\mu_j^0,\mu_j^1,\cdots,\mu_j^k)$ with $\mu_j^l=(\int_{I_j}(\phi_j^l)^2dx)^{-1}$. In addition, introduce the vectors
\[\textbf{u}_j=\bigg(u_j^0,u_j^1\cdots,u_j^k\bigg)^T, 
\textbf{A}_j=\bigg(\mathcal{A}_j(\eta_h;\phi_j^0),\mathcal{A}_j(\eta_h;\phi_j^1)
,\cdots,\mathcal{A}_j(\eta_h;\phi_j^k)\bigg)^T,~\Phi_j=\bigg(\phi_j^0,\phi_j^1,\cdots,\phi_j^k\bigg)^T,\]
and the equation \eqref{system} can be rewritten as a linear system 
\begin{equation}\label{linearsystem}
  L_jd\textbf{u}_j=\textbf{A}_jdt+\int_{I_j}\lambda_2\Phi_jdW_tdx,
\end{equation}
which leads to
\begin{equation}\label{du}
  d\textbf{u}_j=L_j^{-1}\textbf{A}_jdt+L_j^{-1}\int_{I_j}\lambda_2\Phi_jdW_tdx.
\end{equation}
Therefore, we have
\begin{eqnarray}
  \int_{I_j}(du_h)^2dx
  &=&\int_{I_j}\Big(\sum^k_{l=0}du_j^l\phi_j^l\Big)\Big(\sum^k_{l=0}du_j^l\phi_j^l\Big)dx
	=\sum^k_{l=0}\bigg(\int_{I_j}\phi_j^l\phi_j^ldx\bigg)(du_j^l)^2 \notag \\
&=&L_jd\textbf{u}_j\cdot 
	d\textbf{u}_j=\Big(\int_{I_j}\lambda_2\Phi_jdW_tdx\Big)\cdot L_j^{-1}\Big(\int_{I_j}\lambda_2\Phi_jdW_tdx\Big)\notag\\
  &=&\lambda_2^2\sum_{l=0}^k \mu_j^l\Big(\int_{I_j}\phi_j^ldW_tdx\Big)^2. \label{718663}
\end{eqnarray}
The definition of $W_t$ yields 
\begin{eqnarray*}
  \Big(\int_{I_j}\phi_j^ldW_tdx\Big)^2=\Big(\int_{I_j}\phi_j^l\sum^\infty_{m=1}\sqrt{\gamma_m}e_m(x)d\mathcal{B}_mdx\Big)^2
  =\sum^\infty_{m=1} \Big(\int_{I_j}\phi_j^l\sqrt{\gamma_m}e_mdx\Big)^2dt.
\end{eqnarray*}
Similarly, we have
\begin{equation}
\int_{I_j} (d\eta_h)^2dx=\lambda_1^2\sum_{l=0}^k \mu_j^l\Big(\int_{I_j}\phi_j^ldW_tdx\Big)^2.
\label{718663_1}
\end{equation}

The combination of \eqref{Ito}, \eqref{qqq1}, \eqref{718663} and \eqref{718663_1} leads to
\begin{align}
 \int_{I_j} (d(u_h)^2 + d(\eta_h)^2) dx =& -2\Theta_{j+\frac12}+2\Theta_{j-\frac12} \notag\\
 &+2\int_{I_j}(\lambda_2u_h-\lambda_1\eta_h)dW_tdx + (\lambda_1^2+\lambda_2^2)\sum_{l=0}^k \mu_j^l\sum^\infty_{m=1} \Big(\int_{I_j}\phi_j^l\sqrt{\gamma_m}e_mdx\Big)^2dt.
 \label{added1}
\end{align}
Summing over all the cells and integrating in time from $0$ to $t$, we can obtain \eqref{1denergy11}. Note that $\int_I(\lambda_2u_h-\lambda_1\eta_h)dW_sdx$ is an It\^{o} integral, thus 
$\mathbb{E}\Big(\int^t_0\int_I(\lambda_2u_h-\lambda_1\eta_h)dW_sdx\Big)=0$, and taking the expectation \eqref{1denergy11} leads to \eqref{1denergy11_1}. 

(b): For the constant $K$ defined in \eqref{added2}, we can bound it by
\begin{align*}
  K\le & \sum_{j=1}^N\sum_{l=0}^k \mu_j^l\sum^\infty_{m=1} 
  \int_{I_j}(\phi_j^l)^2dx\int_{I_j}{\gamma_m}e_m^2dx
  =(k+1)\sum^N_{j=1}\sum^\infty_{m=1} 
\int_{I_j}{\gamma_m}e_m^2dx\\
&=(k+1)\sum^\infty_{m=1} 
\int_{I}{\gamma_m}e_m^2dx=(k+1)Tr(Q),
\end{align*}
where $\mu_j^l=(\int_{I_j}(\phi_j^l)^2dx)^{-1}$ is used.

Next, we show that $K$ is an approximation of $Tr(Q)$. Note that $\phi_j^l=1$, $\mu_j^1 =1/h_j$, and $\int_{I_j}\phi_j^l dx=0$ for $l\ge 1$. 
Let us define $\overline{e}_m=\frac{1}{h_j} \int_{I_j} e_m dx$, and rewrite $K$ as 
\begin{eqnarray}\label{esC}
  K=\sum^N_{j=1}\sum^\infty_{m=1}\frac{\gamma_m}{h_j}\Big(\int_{I_j}e_mdx\Big)^2 +
  \sum^N_{j=1}\sum^\infty_{m=1}\sum_{l=1}^k\mu_j^l \gamma_m \Big(\int_{I_j} \phi_j^l 
  (e_m-\overline{e}_m) dx\Big)^2.
\end{eqnarray}
Using the fact that $Tr(Q)=\sum^\infty_{m=1}\gamma_m$ and $\sum^N_{j=1} \int_{I_j}e_m^2dx = 1$, we have
\begin{align}\notag
  K-Tr(Q)&=\sum^N_{j=1}\sum^\infty_{m=1}\gamma_m\Bigg(\frac{1}{h_j}\Big(\int_{I_j}e_mdx\Big)^2-\int_{I_j}e_m^2dx\Bigg)
+  \sum^N_{j=1}\sum^\infty_{m=1}\sum_{l=1}^k\mu_j^l \gamma_m \Big(\int_{I_j} \phi_j^l 
  (e_m-\overline{e}_m) dx\Big)^2 \\
 & =: \RNum{1}+\RNum{2}. \label{CandTr}
\end{align}
Notice that 
\begin{eqnarray*}
  \int_{I_j}e_m^2 dx=\int_{I_j}(e_m-\overline{e}_m+\overline{e}_m)^2dx
  =\int_{I_j}(e_m-\overline{e}_m)^2dx+\int_{I_j}\overline{e}_m^2 dx=
  \int_{I_j}(e_m-\overline{e}_m)^2dx+\frac1h\Big(\int_{I_j}e_mdx\Big)^2,
\end{eqnarray*}
which leads to
\begin{eqnarray*}
  \RNum{1}=-\sum^N_{j=1}\sum^\infty_{m=1}\gamma_m\int_{I_j}(e_m-\overline{e}_m)^2dx.
\end{eqnarray*}
Since $e_m$ is uniformly bounded, there exists some constant $M$ such that $|e_m-\overline{e}_m|^{2-\alpha}\le M$, and we have the following estimate
\begin{align*}
  \RNum{1}&=-\sum^N_{j=1}\sum^\infty_{m=1}\gamma_m\int_{I_j}(e_m-\overline{e}_m)^2dx \\
    &\le M\sum^N_{j=1}\sum_{m=1}^\infty\gamma_m\int_{I_j}|e_m-\overline{e}_m|^\alpha dx
  \le M\sum^N_{j=1}\sum_{m=1}^\infty\gamma_m (K_m)^\alpha h^{1+\alpha}	  \\
 &= M \Big(\sum^N_{j=1} h^{1+\alpha} \Big) \Big(\sum_{m=1}^\infty\gamma_m (K_m)^\alpha\Big) = O(h^{\alpha}),
\end{align*}
where $K_m=\norm{(e_m)_x}_{{\infty}}$ and the assumption $\sum_{m=1}^\infty\gamma_m (K_m)^\alpha<\infty$ is used.
For the other term $\RNum{2}$, we can apply Young's inequality and follow the similar analysis as above to obtain
\begin{align*}
   \RNum{2}&=\sum^N_{j=1}\sum^\infty_{m=1}\sum_{l=1}^k\mu_j^l \gamma_m \Big(\int_{I_j} \phi_j^l 
  (e_m-\overline{e}_m) dx\Big)^2
  \le \sum^N_{j=1}\sum^\infty_{m=1}\sum_{l=1}^k \mu_j^l \gamma_m \int_{I_j} (\phi_j^l )^2dx \int_{I_j}(e_m-\overline{e}_m)^2dx\\
  &= \sum^N_{j=1}\sum^\infty_{m=1}\sum_{l=1}^k\gamma_m\int_{I_j}(e_m-\overline{e}_m)^2dx
  = k\sum^N_{j=1}\sum^\infty_{m=1}\gamma_m\int_{I_j}(e_m-\overline{e}_m)^2dx =  O(h^{\alpha}).
\end{align*}
The combination of these results lead to the conclusion that
$K-Tr(Q)=O(h^\alpha)$, which finishes the proof.
\end{proof}

\begin{remark}\label{remark1}
  If $W_t$ is the standard Brownian motion, we have $e_1=1$ and $e_m=0$ for $m>1$, therefore it can be easily shown that $\RNum{1}=\RNum{2}=0$ which leads to $K= Tr(Q)$. This means that the continuous energy law \eqref{thm2.9} is exactly preserved by the proposed method. 
\end{remark}

\subsection{Optimal error estimates}
The optimal error estimate analysis of the proposed semi-discrete DG method will be provided in this subsection.

We firstly define the generalized Radau $\mathcal{P}^\alpha$ projection operators which will be used in the analysis. On any cell $I_j$ and for any function $g(x)$, its projection $\mathcal{P}^\alpha g$ into the space $V_h^k$ is given by
\[\int_{I_j}(\mathcal{P}^\alpha g-g(x))v(x)dx=0,~\forall v(x)\in P^{k-1}(I_j),~~\text{and} ~
(\{\mathcal{P}^\alpha g\}+\alpha[\mathcal{P}^\alpha g])_{j+\frac12}=g(x_{j+\frac12}).\]
The following property on the projection error is studied in \cite{XM 2016,SX2021} and will be used throughout this section. 
\begin{lemma}[\bf{Projection error}]
  Let $\mathcal{P}^\alpha$ (with $\alpha\neq 0$) be the generalized Radau projection defined above. There exists some constant $C$ which is independent of $h$, such that
  \[\norm{\mathcal{P}^{\alpha}g-g}\le Ch^{k+1}.\]
\end{lemma}
\begin{theorem}[\bf{Optimal error estimate}] \label{error estimate}
  Let $u$, $\eta\in L^2(\Omega\times[0,T];H^{k+2})$ be the strong solutions to the one-dimensional stochastic Maxwell equations with additive noise \eqref{1}, and $u_h,~\eta_h\in{V}^k_h$ be the numerical solutions given by the DG scheme \eqref{phi} and \eqref{phii}. With the initial conditions chosen as
$$\eta_h(x,0)=\mathcal{P}^{-\alpha}\eta(x,0), ~u_h(x,0)=\mathcal{P}^{\alpha} u(x,0), $$
there holds the following error estimates
  \begin{eqnarray}\label{thm2}
      \norm{u-u_h}^2+\norm{\eta-\eta_h}^2 \le Ch^{2k+2}.
  \end{eqnarray}
\end{theorem}
\begin{proof}
  Note that both the exact solution $\eta$ and the numerical solution $\eta_h$ satisfy the equation \eqref{phi}, and both $u$ and $u_h$ 
  satisfy \eqref{phii}, therefore we have the following error equations
  \begin{eqnarray}
  \int_{I_j}d(\eta-\eta_h)\varphi(x)dx=\Big(\int_{I_j}(u-u_h)\varphi_xdx 
  -((u-\widecheck{u}_h)\varphi^-)_{j+\frac12}+((u-\widecheck{u}_h)\varphi^+)_{j-\frac12}\Big)dt,\label{phie}\\
    \int_{I_j}d(u-u_h)\widetilde{\varphi}(x)dx=\Big(\int_{I_j}(\eta-\eta_h)\widetilde{\varphi}_xdx 
  -((\eta-\widecheck{\eta}_h)\widetilde{\varphi}^-)_{j+\frac12}+((\eta-\widecheck{\eta}_h)\widetilde{\varphi}^+)_{j-\frac12}\Big)dt.\label{phiie}
  \end{eqnarray}
  Let 
  \begin{eqnarray*}
    \xi^\eta=\mathcal{P}^{-\alpha}\eta-\eta_h,~~\epsilon^\eta=\mathcal{P}^{-\alpha}\eta-\eta,~~
    \xi^u=\mathcal{P}^\alpha u-u_h,~~\epsilon^u=\mathcal{P}^\alpha u-u,
  \end{eqnarray*}
  so that we can decompose the numerical error into two terms
  \begin{equation}\label{ooo}
   \eta-\eta_h=\xi^\eta-\epsilon^\eta,\qquad~u-u_h=\xi^u-\epsilon^u.
  \end{equation} 
By choosing the test functions $\varphi=\xi^\eta$, $\widetilde{\varphi}=\xi^u$ in \eqref{phie} and \eqref{phiie}, and summing up the resulting equations, we obtain
  \begin{align}
          \int_{I_j}(d\xi^\eta\xi^\eta+d\xi^u\xi^u) dx&=\int_{I_j}(d\epsilon^\eta\xi^\eta+d\epsilon^u\xi^u )dx\notag \\       
    &+\Big(\int_{I_j}\xi^u\xi^\eta_xdx
    -\big((\{\xi^u\}+\alpha[\xi^u])(\xi^\eta)^-\big)_{j+\frac12}
    +\big((\{\xi^u\}+\alpha[\xi^u])(\xi^\eta)^+\big)_{j-\frac12}\Big)dt \notag \\
    &-\Big(\int_{I_j}\epsilon^u\xi^\eta_xdx
     -\big((\{\epsilon^u\}+\alpha[\epsilon^u])(\xi^\eta)^-\big)_{j+\frac12}
    +\big((\{\epsilon^u\}+\alpha[\epsilon^u])(\xi^\eta)^+\big)_{j-\frac12}\Big)dt\notag \\
      &+\Big(\int_{I_j}\xi^\eta\xi^u_xdx
    -\big((\{\xi^\eta\}-\alpha[\xi^\eta])(\xi^u)^-\big)_{j+\frac12}
    +\big((\{\xi^\eta\}-\alpha[\xi^\eta])(\xi^u)^+\big)_{j-\frac12}\Big)dt \notag \\ 
    &-\Big(\int_{I_j}\epsilon^\eta\xi^u_xdx
     -\big((\{\epsilon^\eta\}-\alpha[\epsilon^\eta])(\xi^u)^-\big)_{j+\frac12}
    +\big((\{\epsilon^\eta\}-\alpha[\epsilon^\eta])(\xi^u)^+\big)_{j-\frac12}\Big)dt\notag \\
    &= \int_{I_j}(d\epsilon^\eta\xi^\eta+d\epsilon^u\xi^u) dx+ \Big(\widetilde{\Theta}_{j-\frac12}-\widetilde{\Theta}_{j+\frac12}\Big)dt,
        \label{err3}
  \end{align}
  where 
  $\widetilde{\Theta}=\big(\frac12+\alpha\big)(\xi^u)^+(\xi^\eta)^-+\big(\frac12-\alpha\big)(\xi^\eta)^+(\xi^u)^-$,
and the last equality follows from an integration by parts and the definition of $\epsilon^u$ and $\epsilon^\eta$:
  \[\int_{I_j}\epsilon^u\xi^\eta_xdx=\int_{I_j}\epsilon^\eta\xi^u_xdx=(\{\epsilon^u\}+\alpha[\epsilon^u])_{j\pm\frac12}=
  (\{\epsilon^\eta\}-\alpha[\epsilon^\eta])_{j\pm\frac12}=0.\]

By It\^{o}'s lemma, we have
\begin{eqnarray}\label{ito2}
 d(\xi^\eta)^2=2d\xi^\eta\xi^\eta+(d\xi^\eta)^2,~
 d(\xi^u)^2=2d\xi^u\xi^u+(d\xi^u)^2.
\end{eqnarray}
Note that 
\[d(\mathcal{P}^\alpha u)=\mathcal{P}^\alpha(du)=\mathcal{P}^\alpha(-\eta_xdt+\lambda_2dW_t)=
\mathcal{P}^\alpha(-\eta_xdt)+\lambda_2\mathcal{P}^\alpha(dW_t).\]
For any test function $\widetilde{\varphi}$, we have
\begin{eqnarray}\label{proj}
  \int_{I_j}(d\mathcal{P}^\alpha u)\widetilde{\varphi}dx
  =\int_{I_j}\mathcal{P}^\alpha(-\eta_x)\widetilde{\varphi}dxdt+\int_{I_j}\lambda_2\widetilde{\varphi}
  \mathcal{P}^\alpha(dW_t)dx.
\end{eqnarray}
Subtracting \eqref{phii} from \eqref{proj}, we obtain 
\begin{eqnarray*}
  \int_{I_j}d\xi^u\widetilde{\varphi}dx=\Big(\int_{I_j}(-\eta_h\widetilde{\varphi_x}+\mathcal{P}^\alpha
  (-\eta_x)\widetilde{\varphi})
  dx+(\eta_h^+\widetilde{\varphi}^-)_{j+\frac12}-(\eta_h^+\widetilde{\varphi}^+)_{j-\frac12}
  \Big)dt+\lambda_2\int_{I_j}(\mathcal{P}^\alpha(dW_t)-dW_t)\widetilde{\varphi}dx.
\end{eqnarray*}
By choosing $\widetilde{\varphi}=d\xi^u$, the first term on the right-hand side becomes zero due to It\^{o}'s calculus, and we have
\[\int_{I_j}(d\xi^u)^2 dx=\lambda_2\int_{I_j}(\mathcal{P}^\alpha(dW_t)-dW_t)d\xi^udx\le \frac12\int_{I_j}(d\xi^u)^2dx+C
\int_{I_j}(\mathcal{P}^\alpha(dW_t)-dW_t)^2dx,\]
which leads to 
\begin{equation} \label{added10}
\int_{I_j}(d\xi^u)^2dx\le C\int_{I_j}(\mathcal{P}^\alpha(dW_t)-dW_t)^2dx.
\end{equation}
In the same fashion, we can have
\begin{equation}\label{added11}
\int_{I_j}(d\xi^\eta)^2 dx\le C\int_{I_j}(\mathcal{P}^{-\alpha}(dW_t)-dW_t)^2dx.
\end{equation}
The combination of \eqref{ito2}, \eqref{err3} with \eqref{added10}, \eqref{added11} leads to
\begin{eqnarray}\label{err4}
  \begin{aligned}
  &\frac12\int_{I_j} \big( d(\xi^\eta)^2+d(\xi^u)^2 \big)dx\\
  &= \int_{I_j}(d\epsilon^\eta\xi^\eta + d\epsilon^u\xi^u)dx+\Big(\widetilde{\Theta}_{j-\frac12}-\widetilde{\Theta}_{j+\frac12}\Big)dt+\frac12\int_{I_j}\big( (d\xi^u)^2+(d\xi^\eta)^2 \big)dx\\
  &\le \int_{I_j}(d\epsilon^\eta\xi^\eta + d\epsilon^u\xi^u) dx+\Big(\widetilde{\Theta}_{j-\frac12}-\widetilde{\Theta}_{j+\frac12}\Big)dt+C\int_{I_j}(\mathcal{P}^{\pm\alpha}(dW_t)-dW_t)^2dx.
  \end{aligned}
\end{eqnarray}
Summing over $I_j$ and utilizing the periodic boundary conditions yields 
\begin{align}\label{err5}
  &\frac12\int_{I} \big( d(\xi^\eta)^2+d(\xi^u)^2 \big)dx
  \le \int_{I}(d\epsilon^\eta\xi^\eta + d\epsilon^u\xi^u) dx+C\norm{\mathcal{P}^{\pm\alpha}(dW_s)-dW_s}^2.
\end{align}
By integrating \eqref{err4} from $0$ to $t$ and noting that $\norm{\xi^\eta(x,0)}^2=\norm{\xi^u(x,0)}^2=0$, we have 
\begin{align*}
\norm{\xi^\eta(x,t)}^2+\norm{\xi^u(x,t)}^2
 & \le 2\Big(\int_0^t\int_{I} (d\epsilon^\eta\xi^\eta
  + d\epsilon^u\xi^u)dxds\Big)+C\int^t_0\norm{\mathcal{P}^{\pm\alpha}(dW_s)-dW_s}^2\\
 & \le \Big(\int^t_0\norm{\xi^\eta(x,s)}^2+\norm{\xi^u(x,s)}^2ds\Big)
  +Ch^{2k+2},
\end{align*}
where the projection error is used in the last inequality.
Applying the Gronwall's inequality and combining with the optimal projection error yields the desired optimal error estimate \eqref{thm2}.
\end{proof}

\subsection{Multi-symplectic structure }
The stochastic Maxwell equations \eqref{1} have a multi-symplectic structure, and we will show that the proposed DG method \eqref{phi}-\eqref{phii} preserves the multi-symplecticity. 

Following the idea in \cite{HJZ2014}, we introduce new variables
\[v_t=u,~~\zeta_t=\eta,~~P=u+\frac{1}{2}\zeta_x,~~Q=\eta+\frac{1}{2}v_x,\]
and rewrite \eqref{1} into the following system:
 \begin{equation}\label{2}
  \begin{cases}
\frac12\zeta_x=P-u, &\\
\frac12v_x=Q-\eta,&\\
-dP-\frac12\eta_xdt=-\lambda_2dW_t,&\\
-dQ-\frac12u_xdt=\lambda_1dW_t,&\\
dv=udt,&\\
d\zeta=\eta dt.&
\end{cases}
\end{equation}
Introduce the notation $z=(u,~\eta,~v,~\zeta,~P,~Q)^T$, and the system \eqref{2} can be rewritten as the multi-symplectic 
system
\begin{equation}\label{multi}
  Mdz+Kz_xdt=\nabla_zS_1(z)dt+\nabla_zS_2(z)dW_t,
\end{equation}
where 
\[M=\left(                 
  \begin{array}{cccccc}   
   0&0&0&0&0&0\\  
    0&0&0&0&0&0\\ 
    0&0&0&0&-1&0\\
    0&0&0&0&0&-1\\
    0&0&1&0&0&0\\
    0&0&0&1&0&0
  \end{array}
\right),\qquad K=\left(\begin{array}{cccccc}
0&0&0&\frac12&0&0\\
0&0&\frac12&0&0&0\\
0&-\frac12&0&0&0&0\\
-\frac12&0&0&0&0&0\\
0&0&0&0&0&0\\
0&0&0&0&0&0
\end{array} \right), \]
and 
\[S_1(z)=Pu+Q\eta-\frac{u^2}{2}-\frac{\eta^2}{2},\qquad S_2(z)=\lambda_1\zeta-\lambda_2v.\] 
Applying the exterior derivative to the system \eqref{multi} yields the following variation equation for the one-form $Z$:
\begin{equation}{\label{variation}}
  MdZ+KZ_xdt=\nabla^2S_1(z)Zdt,
\end{equation}
Let $U, V\in\mathbb{R}^d$ be any solutions to the variation equation \eqref{variation}, and define 
\begin{equation*}
 \omega(U,V)=MU\cdot V, \qquad \kappa(U,V)=KU\cdot V,
\end{equation*}
then the system \eqref{multi} can be shown to satisfy the multi-symplectic conservation law given by
\begin{equation}\label{consmulti}
  d\omega+\kappa_xdt=0.
\end{equation}
Multi-symplectic numerical method refers to the method that satisfies a consistent discrete version of this conservation law.

Since the new system \eqref{2} is equivalent to the original model \eqref{1}, we can rewrite the proposed DG methods \eqref{phi} and \eqref{phii} into a consistent formulation for the system \eqref{2}.
We start by defining the variables $v_h$ and $\zeta_h$ as follows:  for $u_h$ and $\eta_h$ 
defined in \eqref{phi} and \eqref{phii}, 
find $v_h, ~\zeta_h\in V_h^k$, such that for all test functions $\psi, \widetilde{\psi}\in V_h^k$, it holds that
\begin{eqnarray}
  &\int_{I_j}dv_h\psi dx=\int_{I_j}u_h\psi dxdt, \qquad\qquad
\int_{I_j}d\zeta_h\widetilde{\psi}dx=\int_{I_j}\eta_h\widetilde{\psi}dxdt. \label{v}
\end{eqnarray}
Next, find $P_h,~Q_h\in V_h^k$, such that for all $\varphi, \widetilde{\varphi}\in V_h^k$, it holds that
\begin{eqnarray}
  &&\int_{I_j}(P_h-u_h)\varphi dx=-\frac12\Big(\int_{I_j}
   \zeta_h\varphi_xdx-(\widehat{\zeta}_h\varphi^-)_{j+\frac12}
   +(\widehat{\zeta}_h\varphi^+)_{j-\frac12}\Big), \label{Ph} \\
     &&\int_{I_j}(Q_h-\eta_h)\widetilde{\varphi}dx=-\frac12\Big(\int_{I_j}
   v_h\widetilde{\varphi}_xdx-(\widehat{v}_h\widetilde{\varphi}^-)_{j+\frac12}
   +(\widehat{v}_h\widetilde{\varphi}^+)_{j-\frac12}\Big) \label{Qh},
\end{eqnarray}
where $\widehat{v}_h=\{v_h\}+2n[v_h],~
  \widehat{\zeta}_h=[\zeta_h]+2m[\zeta_h]$ for some $m,~n\in\mathbb{R}$ with $n-m=\alpha$. 
By combining the derivative of \eqref{Ph} and \eqref{Qh} with the equations \eqref{phii} and \eqref{phi}, we obtain
  \begin{eqnarray}
    &&\int_{I_j}dP_h\phi dx=\frac12\Big(\int_{I_j}\eta_h\phi_xdx-
   (\widehat{\eta}_h\phi^-)_{j+\frac12}
   +(\widehat{\eta}_h\phi^+)_{j-\frac12}\Big)dt+\int_{I_j}\lambda_2\phi dW_tdx
   \label{PPh},\\
   &&\int_{I_j}dQ_h\widetilde{\phi}dx=\frac12\Big(\int_{I_j}
   u_h\widetilde{\phi}_xdx-(\widehat{u}_h\widetilde{\phi}^-)_{j+\frac12}
+(\widehat{u}_h\widetilde{\phi}^+)_{j-\frac12}\Big)dt-\int_{I_j}\lambda_1
\widetilde{\phi}dW_tdx\label{QQh},
  \end{eqnarray}
  where $\widehat{u}_h=\{u_h\}-2m[u_h],~\widehat{\eta}_h=\{\eta_h\}-2n[\eta_h]$. 
  Combining \eqref{v}-\eqref{QQh}, we have derived the following expression: 
    find $P_h,~Q_h,~u_h,~\eta_h,~v_h,~\zeta_h~\in V_h^k$ such that 
  \begin{eqnarray}
   \begin{aligned}
        &\int_{I_j}(P_h-u_h)\varphi dx=-\frac12\Big(\int_{I_j}
   \zeta_h\varphi_xdx-(\widehat{\zeta}_h\varphi^-)_{j+\frac12}
   +(\widehat{\zeta}_h\varphi^+)_{j-\frac12}\Big), \label{DG1} \\
     &\int_{I_j}(Q_h-\eta_h)\widetilde{\varphi}dx=-\frac12\Big(\int_{I_j}
   v_h\widetilde{\varphi}_xdx-(\widehat{v}_h\widetilde{\varphi}^-)_{j+\frac12}
   +(\widehat{v}_h\widetilde{\varphi}^+)_{j-\frac12}\Big), \label{DG2}\\
   &\int_{I_j}dP_h\phi dx=\frac12\Big(\int_{I_j}\eta_h\phi_xdx-
   (\widehat{\eta}_h\phi^-)_{j+\frac12}
   +(\widehat{\eta}_h\phi^+)_{j-\frac12}\Big)dt+\int_{I_j}\lambda_2\phi dW_tdx,
   \label{DG3}\\
   &\int_{I_j}dQ_h\widetilde{\phi}dx=\frac12\Big(\int_{I_j}
   u_h\widetilde{\phi}_xdx-(\widehat{u}_h\widetilde{\phi}^-)_{j+\frac12}
+(\widehat{u}_h\widetilde{\phi}^+)_{j-\frac12}\Big)dt-\int_{I_j}\lambda_1
\widetilde{\phi}dW_tdx,\label{DG4}\\
&\int_{I_j}dv_h\psi dx=\int_{I_j}u_h\psi dxdt,\label{DG5}\\
&\int_{I_j}d\zeta_h\widetilde{\psi}=\int_{I_j}\eta_h\widetilde{\psi}dxdt , \label{DG6}
   \end{aligned}
  \end{eqnarray}
  hold for any $\varphi,~\widetilde{\varphi},~\phi,~\widetilde{\phi},~\psi,~\widetilde{\psi}~\in 
  V_h^k$.
  The numerical fluxes take the form 
    \[\widehat{u}_h=\{u_h\}-2m[u_h],~~\widehat{\eta}_h=\{\eta_h\}-2n[\eta_h],
    ~~\widehat{v}_h=\{v_h\}+2n[v_h],~~\widehat{\zeta}_h=[\zeta_h]+2m[\zeta_h].\]
Note that the method \eqref{DG1} can be rewritten into the corresponding DG scheme for the new system \eqref{multi}: find $z_h\in (V_h^k)^6$, such that for all  $\bm{\varphi}\in (V^k_h)^6$, it holds that 
  \begin{eqnarray}\label{DG}
  \int_{I_j}Mdz_h\cdot \bm{\varphi} dx-\Big(\int_{I_j}Kz_h\cdot\bm{\varphi}_xdx-
  \big(\widehat{Kz_h}\cdot\bm{\varphi}^-\big)_{j+\frac12}+
  \big(\widehat{Kz_h}\cdot\bm{\varphi}^+\big)_{j-\frac12}\Big)dt \notag \\
  =\int_{I_j}\nabla S_1(z_h)\cdot \bm{\varphi} dxdt+\int_{I_j}\nabla S_2(z_h)\cdot 
  \bm{\varphi} dW_tdx,
\end{eqnarray}
where $\widehat{Kz_h}$ is the numerical flux in the form
\[\widehat{Kz_h}=K\{z_h\}+A[z_h],
\qquad   
A= \left(\begin{array}{cccccc}
0&0&0&m&0&0\\
0&0&n&0&0&0\\
0&n&0&0&0&0\\
m&0&0&0&0&0\\
0&0&0&0&0&0\\
0&0&0&0&0&0
\end{array} \right). \]
Applying the exterior derivative to this scheme \eqref{DG} leads to the variational equation
\begin{eqnarray}\label{varia}
  \int_{I_j}MdZ_h\cdot \varphi dx-\Big(\int_{I_j}KZ_h\cdot\varphi_xdx-
  \big(\widehat{KZ_h}\cdot\varphi^-\big)_{j+\frac12}+
  \big(\widehat{KZ_h}\cdot\varphi^+\big)_{j-\frac12}\Big)dt 
  =\int_{I_j}\nabla^2 S_1(z)Z_h\cdot \varphi dxdt,
  \end{eqnarray}

Following the proof of multi-symplecticity of DG methods in \cite{SX2020}, we have the following results.
  \begin{lemma}\label{lemma1}
    For any $U_h, V_h\in {V}^k_h$, we have 
    \begin{eqnarray}\label{lemma1-1}
      KU^-_h\cdot V_h^--\widehat{KU_h}\cdot V_h^-+
      \widehat{KV_h}\cdot U_h^-=
      KU^+_h\cdot V_h^+-\widehat{KU_h}\cdot V_h^++
      \widehat{KV_h}\cdot U_h^+=\mathcal{F}_K(U_h,V_h),
    \end{eqnarray}
    where 
    \[\mathcal{F}_K(U_h,V_h)=\{KU_h\cdot V_h\}-\widehat{KU_h}\cdot \{V_h\}+
    \widehat{KV_h}\cdot \{U_h\}.\]
  \end{lemma}
  
  \begin{theorem}[{\bf Conservation of multi-symplecticity}]\label{conservmulti}
    Let  $U_h, V_h$ be any solutions to the variational equation \eqref{varia}, we have the semi-discrete version of 
    the multi-symplectic conservation laws
    \begin{equation}\label{dismul}
      \int_{I_j} d(MU_h\cdot V_h)dx-\Big(\mathcal{F}_K(U_h,V_h)_{j+\frac12}
      -\mathcal{F}_K(U_h,V_h)_{j-\frac12}\Big)dt=0.
    \end{equation}
  \end{theorem}
  \begin{proof}
    By It\^{o}'s lemma, we have
    \begin{equation}\label{added12}
	d(MU_h\cdot V_h)=MdU_h\cdot V_h+MU_h\cdot dV_h+MdU_h\cdot dV_h.
    \end{equation}
    Note that $U_h(\omega,x,t), V_h(\omega,x,t)\in{V}^k_h$, which can be rewritten as 
    \[U_h(\omega,x,t)=\sum^k_{l=0}\varphi_j^l(x)U^l_j(\omega,t),\qquad V_h(\omega,x,t)=\sum^k_{l=0}\varphi_j^l(x)V^l_j(\omega,t),\]
    where  $\{\varphi_j^l(x),~l=0,\cdots, k\}$ is a set of basis.
    Therefore one has
    \[dU_h=\sum^k_{l=0}\varphi_j^l(x)dU^l_j,\qquad dV_h=\sum^k_{l=0}\varphi_j^l(x)dV^l_j.\]
In the variational equation \eqref{varia}, we set $dZ_h$ to be $dU_h$, and take the test function $\varphi$ to be $dV_h$. As the second and third terms in \eqref{varia} are both drift terms, we can conclude that $\int_{I_j}MdU_h\cdot dV_hdx=0$.
By combining \eqref{added12} and \eqref{varia}, and utilizing the fact that $M$ is anti-symmetric and $\nabla^2S_1$ is symmetric, we obtain
    \begin{eqnarray*}
      &&\int_{I_j}d(MU_h\cdot V_h)dx=\int_{I_j}(MdU_h\cdot V_h+MU_h\cdot dV_h ) dx
      =\int_{I_j} (MdU_h\cdot V_h-MdV_h\cdot U_h ) dx\\
      &&=\Bigg(\int_{I_j}KU_h\cdot(V_h)_xdx-\big(\widehat{KU_h}\cdot V_h^-\big)_{j+\frac12}
      +\big(\widehat{KU_h}\cdot 
      V_h^+\big)_{j-\frac12}\Bigg)dt+\int_{I_j}\nabla^2S_1U_h\cdot V_hdxdt\\
      &&-\Bigg(\int_{I_j}KV_h\cdot(U_h)_xdx-\big(\widehat{KV_h}\cdot U_h^-\big)_{j+\frac12}
      +\big(\widehat{KV_h}\cdot U_h^+\big)_{j-\frac12}\Bigg)dt-\int_{I_j}\nabla^2S_1V_h\cdot 
      U_hdxdt \\
      &&=\Big(\big(KU_h^-\cdot V_h^--\widehat{KU_h}\cdot V_h^-+\widehat{KV_h}\cdot U_h^- \big)_{j+\frac12}
      -\big(KU_h^+\cdot V_h^+-\widehat{KU_h}\cdot V_h^++\widehat{KV_h}\cdot U_h^+\big)_{j-\frac12} \Big)dt\\
      &&= \Big(\mathcal{F}_K(U_h,V_h)_{j+\frac12}
      -\mathcal{F}_K(U_h,V_h)_{j-\frac12}\Big)dt,
    \end{eqnarray*}
where the last equality follows from Lemma \ref{lemma1}. This finishes the proof.
  \end{proof}

\section{Two-dimensional stochastic Maxwell equations with additive noise}\label{2dmtsplctc}
\setcounter{equation}{0} \setcounter{figure}{0}\setcounter{table}{0}

In this section, we will present the discontinuous Galerkin methods for two-dimensional stochastic Maxwell equations with additive noise on cartesian meshes, and study the stability, error estimate and multi-symplecticity of the proposed methods. 

The two-dimensional rectangular computational domain is set to be $I\times J$, and we consider the rectangular partition with the cells denoted by $I_i\times J_j=[x_{i-\frac12},x_{i+\frac12}]\times [y_{j-\frac12},y_{j+\frac12}]$ for $i=1,2,\cdots, N_x$ and $j=1,2\cdots, N_y$. 
Let $x_i=\frac12(x_{i-\frac12}+x_{i+\frac12})$, and $y_j=\frac12(y_{j-\frac12}+y_{j+\frac12})$ . Furthermore, we define the mesh size 
in both directions as $h_{x,i}=x_{i+\frac12}-x_{i-\frac12}$, $h_{y,j}=y_{j+\frac12}-y_{j-\frac12}$, with $h_x=\max_i h_{x,i}$, $h_y=\max_j 
h_{y,j}$ and $h=\max(h_x,h_y)$ being the maximum mesh size. Similar to the 
one-dimensional case, we define the two dimensional piecewise polynomial space $\mathbb{V}_h^k$ as follows: 
 \[\mathbb{V}_h^k=\{v(x,y): v|_{I_i\times J_j}\in Q^k(I_i\times J_j)=P^k(I_i)\otimes P^k(J_j),~~i=1,2,\cdots,N_x; j=1,2,\cdots,N_y\}.\]

The two-dimensional stochastic Maxwell equations with additive noise take the form
  \begin{equation}\label{2d}
  \begin{cases}
-dE+T_xdt-S_ydt=\lambda_1dW_t,\\
dS+E_ydt=\lambda_2dW_t,\\
dT-E_xdt=\lambda_2dW_t.
\end{cases}
\end{equation}
The DG scheme for \eqref{2d} is formulated as follows: find $E_h,~S_h,~T_h\in 
\mathbb{V}_h^k$, such that for all test functions $\varphi,~\psi,~\phi\in\mathbb{V}_h^k$, it holds that
 \begin{eqnarray}
   \int_{J_j}\int_{I_i}dE_h\varphi dxdy&=&-\int_{J_j}\Bigg(
   \int_{I_i}T_h\varphi_xdx-\big(\widetilde{(T_h)_{\alpha_1}}\varphi^-\big)_{i+\frac12,y}+
   \big(\widetilde{(T_h)_{\alpha_1}}\varphi^+\big)_{i-\frac12,y}\Bigg)
   dydt \notag \\
   &+&\int_{I_i}\Bigg(\int_{J_j}S_h\varphi_ydy-\big(\widetilde{(S_h)_{-\alpha_2}}\varphi^-\big)_{x,j+\frac12}   
   +\big(\widetilde{(S_h)_{-\alpha_2}}\varphi^+\big)_{x,j-\frac12}\Bigg)dxdt\notag \\
   &-&
   \int_{J_j}\int_{I_i}\lambda_1\varphi dW_tdxdy, \label{2ddgg1}\\
   \int_{J_j}\int_{I_i}dS_h\psi dxdy &=&
   \int_{I_i}\Bigg(\int_{J_j}E_h\psi_ydy-\big(\widetilde{(E_h)_{\alpha_2}}\psi^-\big)_{x,j+\frac12}
     +\big(\widetilde{(E_h)_{\alpha_2}}\psi^+\big)_{x,j-\frac12}\Bigg)dxdt \notag \\
     &+&\int_{J_j}\int_{I_i}
     \lambda_2\psi dW_tdxdy , \label{2ddgg2}\\
     \int_{J_j}\int_{I_i} dT_h \phi dxdy
     &=& -\int_{J_j}\Bigg(\int_{I_i}E_h\phi_x dx -\big(\widetilde{(E_h)_{-\alpha_1}}\phi^-\big)_{i+\frac12,y}
     +\big(\widetilde{(E_h)_{-\alpha_1}}
     \phi^+\big)_{i-\frac12,y}\Bigg)dydt\notag \\
     &+&\int_{J_j}\int_{I_i}\lambda_2
     \phi dW_tdxdy, \label{2ddgg3}
 \end{eqnarray} 
where the generalized alternating numerical fluxes are defined as follows:
\[\widetilde{g_\alpha}=\{g\}+\alpha[g], \qquad \text{for}~ g\in\mathbb{V}_h^k~\text{and}~ \alpha\in\{\pm\alpha_1,\pm\alpha_2\}\subset\mathbb{R}.\]
For the ease of presentation, we also introduce the following operators: for $\alpha\in\mathbb{R}$, 
$f,~g\in\mathbb{V}_h^k$, 
\[\mathcal{A}_{I_i}(f,g;\alpha)=\int_{I_i}fg_xdx-(\widetilde{f_\alpha}g^-)_{i+\frac12,y}+(\widetilde{f_\alpha}g^+)_{i-\frac12,y},\]
 \[\mathcal{A}_{J_j}(f,g;\alpha)=\int_{J_j}fg_ydy-(\widetilde{f_\alpha}g^-)_{x,j+\frac12}+(\widetilde{f_\alpha}g^+)_{x,j-\frac12}.\]
 
\subsection{Discrete energy law}\label{2denergygrowth}
Similar to the 1D case, we start by presenting the linear energy growth property of the exact solutions. 
\begin{theorem}[\bf{Continuous energy law}]\label{2denergypdelevel}
  Let $E,~S,~T$ be the solutions to the equation \eqref{2d} under the periodic boundary condition, and define the two-dimensional energy as $\mathcal{E}(t)=\int_J\int_I E(x,y,t)^2+S(x,y,t)^2+T(x,y,t)^2 
  dxdy$, then for any $t$, the global stochastic energy satisfies the following
energy law
\begin{eqnarray}\label{2dEL}
 \mathcal{E}(t)=\mathcal{E}(0)+2\int_0^t\int_J\int_I \lambda_2(S+T)-\lambda_1 E dW_sdxdy+(\lambda_1^2+2\lambda_2^2)Tr(Q)t,
\end{eqnarray}
and, after taking the expectation,
  \begin{eqnarray}\label{linearpdelevel}
    \mathbb{E}\Big(\mathcal{E}(t)\Big)=\mathcal{E}(0)+(\lambda_1^2+2\lambda_2^2)Tr(Q)t.
  \end{eqnarray}
\end{theorem}
The proof is skipped here as it follows the same analysis as that of Theorem \ref{1denergypdelevel}.
%
%

\begin{theorem}[\bf{Discrete energy law}]\label{2denergy}  
  Let $E_h(\omega,x,y,t)$, $S_h(\omega,x,y,t)$ and $T_h(\omega,x,y,t)$ be the numerical solutions to the DG methods \eqref{2ddgg1} - \eqref{2ddgg3}. \\
(a): For any $t\in[0,T]$, the numerical solutions satisfy the discrete energy law
  \begin{eqnarray}\label{ccccccccxian}
      &&\norm{E_h(\omega,x,y,t)}^2+\norm{S_h(\omega,x,y,t)}^2+\norm{T_h(\omega,x,y,t)}^2 \\
      &&= 
  2\int_0^t\int_{J}\int_{I}\lambda_2(T_h+S_h)-\lambda_1E_hdW_sdxdy+\norm{E_h(x,y,0)}^2+\norm{S_h(x,y,0)}^2
+\norm{T_h(x,y,0)}^2+(\lambda_1^2+2\lambda_2^2)Kt ,
  \notag 
  \end{eqnarray}
and
\begin{align}\notag
      &\mathbb{E}\Big(\norm{E_h(x,y,t)}^2+\norm{S_h(x,y,t)}^2+\norm{T_h(x,y,t)}^2\Big)\\
   &~~~~~~~~~~~~~~~~~~~~ =\norm{E_h(x,y,0)}^2+\norm{S_h(x,y,0)}^2+\norm{T_h(x,y,0)}^2+(\lambda_1^2+2\lambda_2^2)Kt.\label{2denergyy}
  \end{align}
with 
  \begin{equation}
  K=\sum_{i,j}\sum_{l=0}^{k^2+2k} \mu_{i,j}^l\sum^\infty_{m=1} \Big(\int_{J_j}\int_{I_i}\phi_{i,j}^l\sqrt{\gamma_m}e_mdxdy\Big)^2, \label{added2d2}
  \end{equation}
where  $\{\phi^l_{i,j}, i,j=0,\cdots,k\}$ represents the set of Legendre basis over cell $J_j\times I_i$, and 
$\mu_{i,j}^l=(\int_{J_j}\int_{I_i}(\phi_{i,j}^l )^2dxdy)^{-1}$. \\
(b) The constant $K$ is bounded by $(k^2+2k)Tr(Q)$ with $k$ being the polynomial degree of the DG methods. 
Moreover, if there exists some constant $\alpha>0$ such that the series 
$\sum^\infty_{m=1}\gamma_m (K_m)^\alpha<\infty$ with $K_m=\norm{\nabla e_m}_{{\infty}}$, 
we can show that $K=Tr(Q)+O(h^\alpha)$, i.e., $K$ is an approximation of $Tr(Q)$ appearing in the continuous energy law \eqref{2dEL}-\eqref{linearpdelevel}.

\end{theorem}
\begin{proof}
  (a):   By It\^{o}'s lemma, we know that 
  \begin{equation}\label{ito2d}
      d(E_h)^2=2E_hdE_h+(dE_h)^2,~~d(S_h)^2=2S_hdS_h+(dS_h)^2,~~d(T_h)^2=2T_hdT_h+(dT_h)^2.
  \end{equation}
By taking the test functions $\varphi=E_h$, $\psi=S_h$ and  $\phi=T_h$ in \eqref{2ddgg1}-\eqref{2ddgg3}, and adding these equations, we have
  \begin{eqnarray}\label{qqqq}
    &&\int_{J_j}\int_{I_i}(dE_h)E_h+(dS_h)S_h+(dT_h)T_hdxdy\notag\\
    &&=\int_{J_j}-(\Theta)_{i-\frac12,y}+(\Theta)_{i+\frac12,y}dydt
    +\int_{I_i}(\widetilde{\Theta})_{x,j-\frac12}-(\widetilde{\Theta})_{x,j+\frac12}dxdt\notag \\
    &&+\int_{J_j}\int_{I_i}\lambda_2(T_h+S_h)-\lambda_1E_hdW_tdxdy,
  \end{eqnarray}
  where 
  \[\Theta=\big(\frac12+\alpha_1\big)T_h^+E_h^-+\big(\frac12-\alpha_1\big)E_h^+T_h^-,
 ~~ \widetilde{\Theta}=\big(\frac12+\alpha_2\big)E_h^+S_h^-+\big(\frac12-\alpha_2\big)E_h^-S_h^+.\]
  Following the same analysis as in Theorem \ref{1denergy}, we can evaluate the terms 
we have
\begin{align}
  &\int_{J_j}\int_{I_i}(dS_h)^2dxdy= \int_{J_j}\int_{I_i}(dT_h)^2dxdy=\lambda_2^2 \sum_{l=0}^{k^2+2k} \mu_{i,j}^l\sum^\infty_{m=1} \Big(\int_{J_j}\int_{I_i}\phi_{i,j}^l\sqrt{\gamma_m}e_mdxdy\Big)^2dt, \notag \\
  &\int_{J_j}\int_{I_i}(dE_h)^2dxdy =\lambda_1^2 \sum_{l=0}^{k^2+2k} \mu_{i,j}^l\sum^\infty_{m=1} \Big(\int_{J_j}\int_{I_i}\phi_{i,j}^l\sqrt{\gamma_m}e_mdxdy\Big)^2dt,   \label{added13}
  \end{align}
By combining the equations \eqref{ito2d}-\eqref{added13}, summing over all the cells and integrating in time from $0$ to $t$, 
we can obtain 
  \begin{eqnarray}\label{cccccccc}
      &&\norm{E_h(x,y,t)}^2+\norm{S_h(x,y,t)}^2+\norm{T_h(x,y,t)}^2 \\
      &&= 
  2\int_0^t\int_{J}\int_{I}\lambda_2(T_h+S_h)-\lambda_1E_hdW_sdxdy+\norm{E_h(x,y,0)}^2+\norm{S_h(x,y,0)}^2+\norm{T_h(x,y,0)}^2+(\lambda_1^2+2\lambda_2^2)Kt. \notag
  \end{eqnarray}
Note that  $\mathbb{E}\Big(\int_0^t\int_{J}\int_{I}\lambda_2(T_h+S_h)-\lambda_1E_hdW_sdxdy\Big)=0$ since it is an 
It\^{o} integral, therefore taking expectation of \eqref{cccccccc} leads to the discrete energy law \eqref{2denergyy}.

 (b): The estimate of the constant $K$ follows an exact same analysis as in the proof of Theorem \ref{1denergy} (b), that is, 
 \[K\le (k^2+2k)Tr(Q), \]
  and if  $\sum^\infty_{m=1}\gamma_m (K_m)^\alpha <\infty$, we can show that 
    \[K-Tr(Q)=O(h^\alpha).\]
If $W_t$ is the standard Brownian motion, it can be shown that $K= Tr(Q)$ (see Remark \ref{remark1}), 
which means that the two-dimensional continuous energy law \eqref{2dEL} is exactly preserved by the proposed method. 
\end{proof}

\subsection{Optimal error estimate}
In this section we study the convergence rate of the DG scheme \eqref{2ddgg1}-\eqref{2ddgg3}. 
We start with defining the generalized Radau projection in $\mathbb{R}^2$ as 
\begin{eqnarray}\label{2dproj}
   \mathbb{P}^\alpha_x=\mathcal{P}^\alpha_x \otimes \mathcal{P}_y,~\mathbb{P}^\alpha_y=\mathcal{P}_x \otimes \mathcal{P}^\alpha_y,~
 \mathbb{P}^{\alpha,\beta}= \mathcal{P}^\alpha_x \otimes \mathcal{P}^\beta_y,
\end{eqnarray}
where $\mathcal{P} $ is the $L^2$ projection. 
The following lemmas are provided  in \cite{XM 2016} and will be useful in our analysis.
\begin{lemma}[\bf{Superconvergence property}]\label{superconvergence}
For any function $w$, denote $\epsilon=\mathbb{P}^{\alpha,\beta}w-w$, with $\mathbb{P}^{\alpha,\beta}$ being the projection defined in \eqref{2dproj}. For any $v\in \mathbb{V}_h^k$, there exists some constant $C$ such that 
    \[\Bigg|\sum_{i,j}\int_{J_j}\mathcal{A}_{I_i}(\epsilon,v,\alpha)dy\Bigg|\le 
    Ch^{k+1}\norm{v},~~~
  \Bigg|\sum_{i,j}\int_{I_i}\mathcal{A}_{J_j}(\epsilon,v,\alpha)dx\Bigg|\le Ch^{k+1}\norm{v}.\]
 \end{lemma}
 \begin{lemma}[\bf{Projection error}]
   Let $\Pi$ be any projection defined in \eqref{2dproj}. For any function $w(x,y)$, there exists some constant $C$, such that 
   \[\norm{\Pi w-w}\le Ch^{k+1}.\]
 \end{lemma}
 
 Now we turn to the main result on the error estimate of the DG methods.  
\begin{theorem}[\bf{Optimal error estimate}] \label{2derror}
Let $E_h, S_h, T_h\in\mathbb{V}_h^k$ be the numerical solutions given by the DG scheme \eqref{2ddgg1} - \eqref{2ddgg3},
and $E, T, S\in L^2(\Omega\times[0,T];H^{k+2})$ are strong solutions to \eqref{2d}. With the initial conditions chosen as 
\[E_h(x,y,0)=\mathbb{P}^{-\alpha_1,\alpha_2}E(x,y,0);~ S_h(x,y,0)=\mathbb{P}^{-\alpha_2}_yS(x,y,0);~ T_h(x,y,0)=\mathbb{P}_x^{\alpha_1}T(x,y,0),\]
there holds the following error estimates
\begin{equation}\label{2derroreq}
\norm{E-E_h}^2+\norm{S-S_h}^2+\norm{T-T_h}^2\le Ch^{2k+2},
\end{equation}
where the constant C denotes a generic positive constant independent of the spatial step sizes h.
\end{theorem}
\begin{proof}
Since both the numerical and exact solutions satisfy the equations \eqref{2ddgg1} - \eqref{2ddgg3},  we have the error equations
\begin{eqnarray}
   \int_{J_j}\int_{I_i}d(E-E_h)\varphi dxdy&=&-\int_{J_j}\mathcal{A}_{I_i}(T-T_h,\varphi;\alpha_1)
   dydt 
   +\int_{I_i}\mathcal{A}_{J_j}(S-S_h,\varphi;-\alpha_2)dxdt,
  \label{2ddgge1}\\
   \int_{J_j}\int_{I_i}d(S-S_h)\psi dxdy &=&
   \int_{I_i}\mathcal{A}_{J_j}(E-E_h,\psi;\alpha_2)dxdt , 
     \label{2ddgge2}\\
     \int_{J_j}\int_{I_i} d(T-T_h) \phi dxdy
     &=& -\int_{J_j}\mathcal{A}_{I_i}(E-E_h,\phi;-\alpha_1)dydt. 
     \label{2ddgge3}
 \end{eqnarray} 
 Let 
\begin{align*}
&\xi^E=\mathbb{P}^{-\alpha_1,\alpha_2}E-E_h, \quad \xi^S=\mathbb{P}_y^{-\alpha_2}S-S_h, \quad
  \xi^T=\mathbb{P}_x^{\alpha_1}T-T_h,  \\
&\epsilon^E=\mathbb{P}^{-\alpha_1,\alpha_2}E-E, \quad
 \epsilon^S=\mathbb{P}_y^{-\alpha_2}S-S, \quad \epsilon^T=\mathbb{P}_x^{\alpha_1}T-T,  
 \end{align*}
which leads to the decomposition of the error into two terms as
\[E-E_h=\xi^E-\epsilon^E,~~S-S_h=\xi^S-\epsilon^S,~~T-T_h=\xi^T-\epsilon^T.\]

By choosing the test functions as $\varphi=\xi^E,~\psi=\xi^S,~\phi=\xi^T$ in \eqref{2ddgge1}-\eqref{2ddgge3}, and noting that 
\[\int_{J_j}\mathcal{A}_{I_i}(\epsilon^T,\xi^E;\alpha_1)dydt=\int_{I_i}\mathcal{A}_{J_j}(\epsilon^S,\xi^E;-\alpha_2)dxdt=0\] 
by the definition of the projections, we have
\begin{eqnarray}
 &&\int_{J_j}\int_{I_i}d\xi^E\xi^E - d\epsilon^E\xi^Edxdy
 =-\int_{J_j}\mathcal{A}_{I_i}(\xi^T,\xi^E;\alpha_1)dydt
+\int_{I_i}\mathcal{A}_{J_j}(\xi^S,\xi^E;-\alpha_2) dxdt, \label{2derror1}\\
 && \int_{J_j}\int_{I_i}d\xi^S\xi^S - d\epsilon^S\xi^Sdxdy
  =\int_{I_i}\mathcal{A}_{J_j}(\xi^E,\xi^S;\alpha_2)dxdt
  -\int_{I_i}\mathcal{A}_{J_j}(\epsilon^E,\xi^S;\alpha_2)dxdt,   \label{2derror2}\\
 && \int_{J_j}\int_{I_i}d\xi^T\xi^T - d\epsilon^T\xi^T 
  dxdy=
  -\int_{J_j} \mathcal{A}_{I_i}(\xi^E,\xi^T;-\alpha_1)dydt
  +\int_{J_j}\mathcal{A}_{I_i}(\epsilon^E,\xi^T;-\alpha_1)dydt.  \label{2derror3}
\end{eqnarray}
Summing up \eqref{2derror1}-\eqref{2derror3} and using integration by parts, we 
obtain 
\begin{align}
    &\int_{J_j}\int_{I_i}d\xi^E\xi^E+d\xi^S\xi^S+d\xi^T\xi^T dxdy~~~~~~~~~~~~~~~~~~~~~~~~~~  \notag \\
    &=    \int_{J_j}\int_{I_i}d\epsilon^E\xi^E+d\epsilon^S\xi^S+d\epsilon^T\xi^T    dxdy
    -\int_{J_j}\Bigg(\Pi_{i-\frac12,y}-\Pi_{i+\frac12,y}\Bigg)
    dydt+\int_{I_i}\Bigg(\bar{\Pi}_{x,j-\frac12}
    -\bar{\Pi}_{x,j+\frac12}\Bigg)dxdt	\notag \\
    &\quad-\int_{I_i}\mathcal{A}_{J_j}(\epsilon^E,\xi^S;\alpha_2)dxdt
  +\int_{J_j}\mathcal{A}_{I_i}(\epsilon^E,\xi^T;-\alpha_1)dydt, \label{2dqqqq}
\end{align}
where 
\[\Pi=\big(\frac12+\alpha_1\big)(\xi^T)^+(\xi^E)^-+\big(\frac12-\alpha_1\big)(\xi^T)^-(\xi^E)^+,~~ 
\bar{\Pi}=\big(\frac12+\alpha_2\big)(\xi^S)^-(\xi^E)^++ \big(\frac12-\alpha_2\big)(\xi^S)^+(\xi^E)^-.\]

By It\^{o}'s lemma, we have
\[d(\xi^E)^2=2d\xi^E\xi^E+d(\xi^E)^2,~~d(\xi^S)^2=2d\xi^S\xi^S+d(\xi^S)^2,~~d(\xi^T)^2=2d\xi^T\xi^T+d(\xi^T)^2.\]
Following the exact same analysis as shown in the proof of Theorem \ref{error estimate}, we have
\[\int_{J_j}\int_{I_i}d(\xi^E)^2dxdy\le C\int_{J_j}\int_{I_i}(\mathbb{P}^{-\alpha_1,\alpha_2}(dW_t)-dW_t)^2dxdy,\]
\[\int_{J_j}\int_{I_i}d(\xi^S)^2dxdy\le C\int_{J_j}\int_{I_i}(\mathbb{P}^{-\alpha_2}_y(dW_t)-dW_t)^2dxdy,\]
\[\int_{J_j}\int_{I_i}d(\xi^T)^2dxdy\le C\int_{J_j}\int_{I_i}(\mathbb{P}^{\alpha_1}_x(dW_t)-dW_t)^2dxdy.\]
Therefore, by summing over all the cells $I_i\times J_j$ in \eqref{2dqqqq} and applying the superconvergence property in Lemma \ref{superconvergence}, we obtain 
\begin{align}\label{2dwwwww}
  \begin{aligned}
  &\frac12\int_J\int_Id(\xi^E)^2+d(\xi^S)^2+(d\xi^T)dxdy\\
  &=\frac12\int_{J}\int_{I}(d\xi^E)^2+(d\xi^S)^2+(d\xi^T)^2dxdy
  + \int_{J}\int_{I}d\xi^E\xi^E+d\xi^S\xi^S+d\xi^T\xi^T dxdy\\
  &  \le  \int_{J}\int_{I}d\epsilon^E\xi^E+d\epsilon^S\xi^S+d\epsilon^T\xi^T 
    dxdy+Ch^{k+1}\big(\norm{\xi^T}+\norm{\xi^S}\big)dt\\
      &+C\big(\norm{\mathbb{P}^{-\alpha_1,\alpha_2}(dW_t)-dW_t}^2
      +\norm{\mathbb{P}^{\alpha_1}_x(dW_t)-dW_t}^2+\norm{\mathbb{P}^{-\alpha_2}_y(dW_t)-dW_t}^2).
    \end{aligned}
\end{align}
Integrating \eqref{2dwwwww} over $t$,  applying the projection error property and Young's inequality, we can show that 
\[\norm{\xi^E}^2+\norm{\xi^S}^2+\norm{\xi^T}^2\le C\int_0^t \norm{\xi^E(x,y,s)}^2+\norm{\xi^S(x,y,s)}^2+\norm{\xi^T(x,y,s)}^2ds+ Ch^{2k+2}.\]
Applying the Gronwall's inequality and combining with the optimal projection error yields the desired optimal error estimate \eqref{2derroreq}.
\end{proof}

\subsection{Multi-symplectic structure }

Similar to the 1D case, to rewrite the two-dimensional stochastic Maxwell equations \eqref{2d} in the multi-symplectic form, we introduce the new variables such that 
\[u_t=E,~v_t=S,~w_t=T,~P=T-\frac12u_x,~Q=S+\frac12u_y,~R=E-\frac12w_x+\frac12v_y.\]
and the system \eqref{2d} becomes
\begin{equation}\label{2dre}
  \begin{cases}
-\frac12u_x=P-T,\\
\frac12u_y=Q-S,\\
-\frac12w_x+\frac12v_y=R-E,\\
-dP+\frac12E_xdt=-\lambda_2dW_t,\\
-dQ-\frac12E_ydt=-\lambda_2dW_t,\\
-dR+\frac12T_xdt-\frac12S_ydt=\lambda_1dW_t,\\
dw=Tdt,\\
dv=Sdt,\\
du=Edt.
\end{cases}
\end{equation}
We set $z=(T,S,E,w,v,u,P,Q,R)^T$, and define 
\[M=\left(                 
  \begin{array}{ccccccccc}   
0&0&0&0&0&0&0&0&0\\
0&0&0&0&0&0&0&0&0\\
0&0&0&0&0&0&0&0&0\\
0&0&0&0&0&0&-1&0&0\\
0&0&0&0&0&0&0&-1&0\\
0&0&0&0&0&0&0&0&-1\\
0&0&0&1&0&0&0&0&0\\
0&0&0&0&1&0&0&0&0\\
0&0&0&0&0&1&0&0&0\\
  \end{array}
\right),\]
\[ K_1= \left(                 
  \begin{array}{ccccccccc}   
0&0&0&0&0&-\frac12&0&0&0\\
0&0&0&0&0&0&0&0&0\\
0&0&0&-\frac12&0&0&0&0&0\\
0&0&\frac12&0&0&0&0&0&0\\
0&0&0&0&0&0&0&0&0\\
\frac12&0&0&0&0&0&0&0&0\\
0&0&0&0&0&0&0&0&0\\
0&0&0&0&0&0&0&0&0\\
0&0&0&0&0&0&0&0&0\\
  \end{array}
\right), \qquad K_2=\left(                 
  \begin{array}{ccccccccc}   
0&0&0&0&0&0&0&0&0\\
0&0&0&0&0&\frac12&0&0&0\\
0&0&0&0&\frac12&0&0&0&0\\
0&0&0&0&0&0&0&0&0\\
0&0&-\frac12&0&0&0&0&0&0\\
0&-\frac12&0&0&0&0&0&0&0\\
0&0&0&0&0&0&0&0&0\\
0&0&0&0&0&0&0&0&0\\
0&0&0&0&0&0&0&0&0\\
  \end{array}
\right),\]
\[S_1=PT+QS+RE-\frac{T^2+S^2+E^2}{2},~S_2=\lambda_1u-\lambda_2(w+v), \]
therefore, the two-dimensional stochastic Maxwell equations \eqref{2d} can be rewritten in the following multi-symplectic system
\begin{equation}\label{2dmulti}
  Mdz+K_1z_xdt+K_2z_ydt=\nabla S_1(z)dt+\nabla S_2(z)dW_t.
\end{equation}
Its variational equation takes the form
\begin{equation}\label{2dvar}
  MdZ+K_1Z_xdt+K_2Z_ydt=\nabla^2S_1(z)Zdt.
\end{equation}
Let $U,~V$ be any solution to the variational equation \eqref{2dvar}, and define $\omega=MU\cdot V, ~\kappa_x=K_1U\cdot V,~\kappa_y=K_2U\cdot V$, then we can derive the multi-symplectic conservation law given by
\[d\omega +\kappa_xdt+\kappa_ydt=0.\]

Since the new system \eqref{2dre} is equivalent to the original model \eqref{2d}, we can rewrite the proposed DG methods \eqref{2ddgg1}-\eqref{2ddgg3} into a consistent formulation for the system \eqref{2dre}. 
For $E_h,~S_h,~T_h$ defined in \eqref{2ddgg1} - \eqref{2ddgg3}, find $w_h,~v_h,~u_h\in\mathbb{V}_h^k$, such that for all 
$\psi,~\bar{\psi},~\widetilde{\psi}\in\mathbb{V}_h^k$, it holds that 
\begin{align}
  &\int_{J_j}\int_{I_i}dw_h\psi dxdy=\int_{J_j}\int_{I_i}T_h\psi dxdy\, dt,\label{wh}\\
    &\int_{J_j}\int_{I_i}dv_h\bar{\psi }dxdy=\int_{J_j}\int_{I_i}S_h\bar{\psi }dxdy\, dt,\label{vh}\\
      &\int_{J_j}\int_{I_i}du_h\widetilde{\psi }dxdy=\int_{J_j}\int_{I_i}E_h\widetilde{\psi }dxdy\, dt.\label{uh}
\end{align}
Next, find $P_h,~Q_h,~R_h\in\mathbb{V}_h^k$ such that for any $\varphi,~\widetilde{\varphi},~\bar{\varphi}\in\mathbb{V}_h^k$, it holds that
\begin{eqnarray}
  &&\int_{J_j}\int_{I_i}(P_h-T_h)\varphi dxdy=\frac12\int_{J_j}\Bigg(\int_{I_i}u_h\varphi_xdx-(\widehat{u}_h\varphi^-)_{i+\frac12,y}
   +(\widehat{u}_h\varphi^+)_{i-\frac12,y}\Bigg)dy, \label{P}\\
      &&\int_{J_j}\int_{I_i}(Q_h-S_h)\bar{\varphi}dxdy=-\frac12\int_{I_i}\Bigg(\int_{J_j}u_h\bar{\varphi}_ydy-(\widehat{u}_h\bar{\varphi}^-)_{x,j+\frac12}
   +(\widehat{u}_h{\bar{\varphi}}^+)_{x,j-\frac12}\Bigg)dx, \label{Q}\\
    &&\int_{J_j}\int_{I_i}(R_h-E_h)\widetilde{\varphi} dxdy=\frac12\int_{J_j}\Bigg(\int_{I_i}w_h\widetilde{\varphi}_xdx-(\widehat{w}_h\widetilde{\varphi}^-)_{i+\frac12,y}
   +(\widehat{w}_h\widetilde{\varphi}^+)_{i-\frac12,y}\Bigg)dy\notag\\
   && ~~~~~~~~~~~~~~~~~~~~~~~~~~~~~~~~~-\frac12\int_{I_i}\Bigg(\int_{J_j}v_h\widetilde{\varphi}_ydy-(\widehat{v}_h\widetilde{\varphi}^-)_{x,j+\frac12}
   +(\widehat{v}_h\widetilde{\varphi}^+)_{x,j-\frac12}\Bigg)dx,\label{R}
\end{eqnarray}
where the numerical fluxes are chosen as 
\[(\widehat{u}_h)_{x_0,y}=(\{u_h\}-2m_1[u_h])_{x_0,y},~(\widehat{w}_h)_{x_0,y}=(\{w_h\}-2n_1[w_h])_{x_0,y},~~m_n,~n_2\in\mathbb{R} ~\text{with}~m_1-n_1=\alpha_1,\]
\[(\widehat{u}_h)_{x,y_0}=(\{u_h\}+2m_2[u_h])_{x,y_0},~(\widehat{v}_h)_{x,y_0}=(\{v_h\}+2n_2[u_h])_{x,y_0},~~m_n,~n_2\in\mathbb{R} ~\text{with}~m_2-n_2=\alpha_2.\]
By combining the derivative of \eqref{P}-\eqref{R} with the equations \eqref{2ddgg3}-\eqref{2ddgg1}, we obtain
\begin{align} \notag
&\int_{J_j}\int_{I_i}dP_h\phi dxdy=-\frac12\int_{J_j}\Bigg(\int_{I_i}E_h\phi_xdx-(\widehat{E}_h\phi^-)_{i+\frac12,y}
   +(\widehat{E}_h\phi^+)_{i-\frac12,y}\Bigg)dy dt
   +\lambda_2 
   \int_{J_j}\int_{I_i}\phi dW_tdxdy, \\
   &\int_{J_j}\int_{I_i}dQ_h\bar{\phi}dxdy=\frac12\int_{I_i}\Bigg(\int_{J_j}E_h\bar{\phi}_ydy-(\widehat{E}_h\bar{\phi}^-)_{x,j+\frac12}
   +(\widehat{E}_h{\bar{\phi}}^+)_{x,j-\frac12}\Bigg)dxdt+\lambda_2 
   \int_{J_j}\int_{I_i}\bar{\phi} dW_tdxdy, \notag \\
   &\int_{J_j}\int_{I_i}dR_h\widetilde{\phi}dxdy=-\frac12\int_{J_j}\Bigg(\int_{I_i}T_h\widetilde{\phi}_xdx-(\widehat{T}_h\widetilde{\phi}^-)_{i+\frac12,y}
   +(\widehat{T}_h\widetilde{\phi}^+)_{i-\frac12,y}\Bigg)dy \label{PQR3}\\
   & ~~~~~~~~~~~~~~~~~~~~~~~+\frac12\int_{I_i}\Bigg(\int_{J_j}S_h\widetilde{\phi}_ydy-(\widehat{S}_h\widetilde{\phi}^-)_{x,j+\frac12}
   +(\widehat{S}_h\widetilde{\phi}^+)_{x,j-\frac12}\Bigg)dx-\lambda_1
   \int_{J_j}\int_{I_i}\widetilde{\phi} dW_tdxdy,  \notag
\end{align}
where the numerical fluxes are  
\[(\widehat{T}_h)_{x_0,y}=(\{T_h\}+2m_1[T_h])_{x_0,y},~~(\widehat{E}_h)_{x_0,y}=(\{E_h\}+2n_1[E_h])_{x_0,y},\]
\[(\widehat{E}_h)_{x,y_0}=(\{E_h\}-2n_2[E_h])_{x,y_0},~~(\widehat{S}_h)_{x,y_0}=(\{S_h\}-2m_2[S_h])_{x,y_0}.\]
Combining \eqref{wh}-\eqref{PQR3}, we have derived the equivalent formulation of the DG scheme \eqref{2ddgg1}-\eqref{2ddgg3} for the new system \eqref{2dmulti}: 
find $z_h\in 
(\mathbb{V}^k_h)^9$, such that for all $\bm{\varphi}\in(\mathbb{V}^k_h)^9$, we have
\begin{eqnarray}
  &&\int_{J_j}\int_{I_i}Mdz_h\cdot \bm{\varphi} dxdy-\int_{J_j}\Big(\int_{I_i}K_1z_h\cdot\bm{\varphi}_xdx
  -\big(\widehat{K_1z_h}\cdot\bm{\varphi}^-\big)_{i+\frac12,y}+
  \big(\widehat{K_1z_h}\cdot\bm{\varphi}^+\big)_{i-\frac12,y}\Big)dydt \notag \\
  &&-\int_{I_i}\Big(\int_{J_j}K_2z_h\cdot\bm{\varphi}_ydy
  -\big(\widehat{K_2z_h}\cdot\bm{\varphi}^-\big)_{x,j+\frac12}+
  \big(\widehat{K_2z_h}\cdot\bm{\varphi}^+\big)_{x,j-\frac12}\Big)dxdt\notag \\
  &&=\int_{J_j}\int_{I_i}\nabla S_1(z_h)\cdot\bm{\varphi} dxdydt+\int_{J_j}\int_{I_i}\nabla S_2(z_h)\cdot\bm{\varphi} dW_tdxdy, \label{2ddg}
\end{eqnarray}
where  $(\widehat{K_1z_h})_{x_0,y}=(K_1\{z_h\}+A_1[z_h])_{x_0,y}$ , and 
$(\widehat{K_2z_h})_{x,y_0}=(K_2\{z_h\}+A_2[z_h])_{x,y_0}$, 
  \[A_1=\left(                 
  \begin{array}{ccccccccc}   
0&0&0&0&0&m_1&0&0&0\\
0&0&0&0&0&0&0&0&0\\
0&0&0&n_1&0&0&0&0&0\\
0&0&n_1&0&0&0&0&0&0\\
0&0&0&0&0&0&0&0&0\\
m_1&0&0&0&0&0&0&0&0\\
0&0&0&0&0&0&0&0&0\\
0&0&0&0&0&0&0&0&0\\
0&0&0&0&0&0&0&0&0\\
  \end{array}\right), \qquad A_2=\left(                 
  \begin{array}{ccccccccc}   
0&0&0&0&0&0&0&0&0\\
0&0&0&0&0&m_2&0&0&0\\
0&0&0&0&n_2&0&0&0&0\\
0&0&0&0&0&0&0&0&0\\
0&0&n_2&0&0&0&0&0&0\\
0&m_2&0&0&0&0&0&0&0\\
0&0&0&0&0&0&0&0&0\\
0&0&0&0&0&0&0&0&0\\
0&0&0&0&0&0&0&0&0\\
  \end{array}
\right),\]
Applying the exterior derivative to this scheme \eqref{2ddg} leads to the variation equation 
\begin{eqnarray}
  &&\int_{J_j}\int_{I_i}MdZ_h\cdot \bm{\varphi} dxdy-\int_{J_j}\Big(\int_{I_i}K_1Z_h\cdot\bm{\varphi} _xdx
  -\big(\widehat{K_1Z_h}\cdot\bm{\varphi} ^-\big)_{i+\frac12,y}+
  \big(\widehat{K_1Z_h}\cdot\bm{\varphi} ^+\big)_{i-\frac12,y}\Big)dydt \label{2ddgvar} \\
  &&-\int_{I_i}\Big(\int_{J_j}K_2z_h\cdot\bm{\varphi} _ydy
  -\big(\widehat{K_2Z_h}\cdot\bm{\varphi} ^-\big)_{x,j+\frac12}+
  \big(\widehat{K_2Z_h}\cdot\bm{\varphi} ^+\big)_{x,j-\frac12}\Big)dxdt=\int_{J_j}\int_{I_i}\nabla^2S_1Z_h\cdot\bm{\varphi}  dxdydt, \notag  
\end{eqnarray}  

\begin{theorem} [\bf{Conservation of multi-symplecticity}]\label{2dconservmulti}
      Let  $U_h, V_h$ be any solutions to the variational equation \eqref{2ddgvar}, we have the semi-discrete version of 
    the multi-symplectic conservation laws
    \begin{eqnarray}
      \int_{J_j}\int_{I_i} d(MU_h\cdot V_h)dx-\int_{J_j}\Big(\mathcal{F}_{K_1}(U_h,V_h)_{i+\frac12,y}
      -\mathcal{F}_{K_1}(U_h,V_h)_{i-\frac12,y}\Big)dydt \notag \\
      -\int_{I_i}\Big(\mathcal{F}_{K_2}(U_h,V_h)_{x,j+\frac12}
      -\mathcal{F}_{K_2}(U_h,V_h)_{x,j-\frac12}\Big)dxdt=0. \label{2ddismul}
    \end{eqnarray}
\end{theorem}
The proof follows the same idea as that of Theorem \ref{conservmulti}, and is skipped here to save space.

\section{Symplectic time discretization}\label{time}
\setcounter{equation}{0} \setcounter{figure}{0}\setcounter{table}{0}

In this section, we discuss the symplectic temporal discretization for the semi-discrete DG methods presented in the previous section. 

In the one-dimensional case, we can set $\eta_h|_{I_j}=\sum^k_{l=1}\eta_j^l\varphi_j^l$, 
~$u_h|_{I_j}=\sum^k_{l=1}u_j^l\varphi_j^l$, with the set $\{\varphi_j^l\}$ being the basis of $V_h^k$, and introduce the notations
\[p=(\eta_1^0,\cdots,\eta_1^k,\eta_2^0,\cdots,\eta_2^k,\cdots,\eta_J^k)^T,
\quad q=(u_1^0,\cdots,u_1^k,u_2^0,\cdots,u_2^k,\cdots,u_J^k)^T.\]
In two-dimensional case, consider 
\[E_h=\sum^k_{l=1} E_{i,j}^l\varphi_{i,j}^l,\quad  S_h=\sum^k_{l=1} S_{i,j}^l\varphi_{i,j}^l, \quad T_h=\sum^k_{l=1} T_{i,j}^l\varphi_{i,j}^l,\]
where $\{\varphi_{i,j}^l\}$ is the basis of $\mathbb{V}_h^k$.  Define 
\begin{align*}
&E_{i,j}=(E_{i,j}^0, E_{i,j}^1, \cdots, E_{i,j}^{k^2+2k})^T, \\
&\bm{E}_h= \left(E_{1,1}, E_{2,1}, \cdots, E_{I,1}, E_{1,2}, \cdots, E_{I,J} \right)^T, 
\end{align*}
and similarly for $\bm{S}_h$, $\bm{T}_h$, and then introduce the notations 
\[p=\bm{E}_h, \qquad q=\left( \bm{S}_h , ~ \bm{T}_h\right)^T.\]
With these notations, either the one-dimensional scheme \eqref{phi}-\eqref{phii} or the one-dimensional method \eqref{2ddgg1}-\eqref{2ddgg3} can be simplified into the following stochastic differential equation: 
\begin{align} \label{tq}
\begin{aligned}
  dp=Aq+Ld\mathcal{B}_t, \\
  dq=Bp+Nd\mathcal{B}_t, 
\end{aligned}  
\end{align}
where $A,~B,~L,~N$ are some constant matrices which may take different values in 1D or 2D setting, and $\mathcal{B}_t$ is a Brownian motion. The system \eqref{tq} will be considered in this section.

\subsection{Symplectic Euler}
We let $0=t_0\leq t_1 \leq \cdots \leq t_N=T$ be a partition of the time interval $[0,T]$. By setting $\tau =t_{k+1}-t_k$, and $\Delta \mathcal{B}=\mathcal{B}_{t_{k+1}}-\mathcal{B}_{t_k}$, the symplectic Euler methods for the system \eqref{tq} are given by 
\begin{align}\label{Eq}
 & p^{k+1}=p^k+\tau Aq^k+L\Delta \mathcal{B},  \quad
 & q^{k+1}=q^k+\tau Bp^{k+1}+N\Delta \mathcal{B}. 
\end{align}
Following the studies in \cite{MRT2002}, we have the following result:
\begin{theorem}
  Symplectic Euler method \eqref{Eq} preserves the symplectic structure, and has the first mean-square order of 
convergence.
\end{theorem}
\subsection{Partitioned Runge-Kutta method}
Consider the two-stage PRK methods \cite{MRT2002} for the system \eqref{tq} of the form
\begin{eqnarray}
  \begin{aligned}\label{prk}
    &\mathcal{Q}_1=q^k+N\Big(J_k+\frac{1}{\sqrt{2}}\Delta \mathcal{B}\Big)\\
    &\mathcal{P}_1=p^k+\frac14\tau A\mathcal{Q}_1
    +L\Big(J_k+\frac{1}{2\sqrt3}\Delta \mathcal{B}\Big),\\
    &\mathcal{Q}_2=q^k+\frac23\tau B\mathcal{P}_1+
    N\Big(J_k-\frac{1}{3\sqrt2}\Delta \mathcal{B}\Big),\\
    &\mathcal{P}_2=p^k+\tau \Big(\frac14A\mathcal{Q}_1+\frac34 A\mathcal{Q}_2\Big)
    +L\Big(J_k-\frac{1}{\sqrt3}\Delta\mathcal{B}\Big),\\
    &p^{k+1}=p^k+L\Delta\mathcal{B}+\tau \Big(\frac14 A\mathcal{Q}_1+\frac34 
    A\mathcal{Q}_2\Big)\\
    &q^{k+1}=q^k+N\Delta\mathcal{B}+\tau  
    \Big(\frac23B\mathcal{P}_1+\frac13B\mathcal{P}_2\Big),
  \end{aligned}
\end{eqnarray}
  where 
  $$\tau =t_{k+1}-t_k, \qquad \Delta \mathcal{B}=\mathcal{B}_{t_{k+1}}-\mathcal{B}_{t_k}, \qquad J_k=\frac{1}{\tau}\int_{t_k}^{t_{k+1}}\mathcal{B}_s-\mathcal{B}_{t_k} ds.$$

In order to analyze the convergence rate of the PRK methods \eqref{prk}, we recall the following theorem in \cite[Theorem 1.1]{NM 1995}.
\begin{theorem}\label{onesteporder}
Consider a general system of stochastic differential equation
   \begin{equation}\label{SDE}
  dX=a(t,X)dt+\sum^m_{r=1}b_r(t,X)d\mathcal{B}_r(t),
\end{equation}
where $X$, $a$ and $b_r$ are column vectors defined on $t\in [t_0, T]$, and $\mathcal{B}_r$ are independent 
standard Brownian motions, and let $\overline{X}_{t,x}(t+\tau)$ be a one-step 
approximation satisfying that for any $t_0\le t\le T-\tau$, 
\begin{align}\label{th11}
 &\Bigg|\mathbb{E}\Big(X_{t,x}(t+\tau)-\overline{X}_{t,x}(t+\tau)\Big)\Bigg|\le K(1+|x|^2)^{1/2}\tau^{p_1}, \\
 &\Bigg(\mathbb{E}\Big|X_{t,x}(t+\tau)-\overline{X}_{t,x}(t+\tau)\Big|^2\Bigg)^{1/2}\le 
   K(1+|x|^2)^{1/2}\tau^{p_2}. \label{th22}
\end{align}
Let $p_2\ge\frac12,~~p_1\ge p_2+\frac12$, then for any $N$ and $k=0,1,\cdots,N$, 
the following result holds:
\begin{eqnarray}\label{thcon}
     \Bigg(\mathbb{E}\Big|X_{t_0,X_0}(t_k)-\overline{X}_{t_0,X_0}(t_k)\Big|^2\Bigg)^{1/2}\le 
   K(1+\mathbb{E}|X_0|^2)^{1/2}\tau^{p_2-\frac12}.
\end{eqnarray}
\end{theorem}
The following Lemma is a direct result of Theorem \ref{onesteporder}, and will be useful in our analysis.
\begin{lemma}\label{orderlemma}
  Let the one-step approximation $\overline{X}_{t,x}(t+\tau)$ satisfies the 
  conditions of Theorem \ref{onesteporder}, and suppose that another one step 
  approximation $\widetilde{X}_{t,x}(t+\tau)$ satisfies 
  \begin{align}\label{lem11}
   & \Bigg|\mathbb{E}\Big(\widetilde{X}_{t,x}(t+\tau)-\overline{X}_{t,x}(t+\tau)\Big)\Bigg|
    =O(\tau^{p_1}), \\
    & \Bigg(\mathbb{E}\Big|\widetilde{X}_{t,x}(t+\tau)-\overline{X}_{t,x}(t+\tau)\Big|^2\Bigg)^{1/2}
    =O(\tau^{p_2}), \label{lem22}
  \end{align}
  with the same $p_1$ and $p_2$ as in Theorem \ref{onesteporder}, then the mean-square order of accuracy for $\widetilde{X}_{t,x}(t+\tau)$ 
is also $p_2-1/2$, same as $\overline{X}_{t,x}(t+\tau)$.
\end{lemma}

  We have the following main result on the convergence rate of the PRK method \eqref{prk}.
\begin{theorem}\label{thmprk}
  The explicit PRK method \eqref{prk} for the system \eqref{tq} preserves symplectic structure and has the mean-square order of 2. 
\end{theorem}
\begin{proof}
The preservation of symplectic structure is shown in \cite[Lemma 4.2]{MRT2002}, here we only prove the correct convergence order.
According to \cite[Section 10.5]{PK 1999}, the following one-step approximation for the system \eqref{tq} has the mean-square order $2$:
  \begin{eqnarray}
    \overline{p}^{k+1}=p^k+\tau A{q}^k+\frac12 
    \tau^2AB{p}^k+\tau AN J_k+L\Delta\mathcal{B},\label{2nd1}\\
    \overline{q}^{k+1}={q}^k+\tau Bp^k+\frac12 \tau^2BA q^k+\tau BL J_k+N\Delta\mathcal{B},\label{znd2}
  \end{eqnarray}
which will be used as a reference method. 

  The PRK method \eqref{prk} can be rewritten as
  \[p^{k+1}=p^k+\tau Aq^k+\frac12\tau^2 ABp^k+\tau AN J_k+L\Delta\mathcal{B}+\frac12\tau^2 ABL\Big(J_k+\frac{1}{2\sqrt3}\Delta \mathcal{B}\Big)
  +O(\tau^3).\]
  Notice that 
  \[\mathbb{E}J_k=\mathbb{E}{\Delta \mathcal{B}}=0,\qquad \mathbb{E}{(J_k)^2}=\frac{\tau}{3},\qquad \mathbb{E}(\Delta \mathcal{B})^2=\tau,\]
which, by Young's inequality, leads to 
  \[\mathbb{E}\Big(J_k+\frac{1}{2\sqrt3}\Delta \mathcal{B}\Big)^2\le C\mathbb{E}(J_k)^2+C\mathbb{E}(\Delta \mathcal{B})^2=O(\tau).\]  
  Therefore, we have 
  \[\Bigg|\mathbb{E}\Big(p^{k+1}-\overline{p}^{k+1}\Big)\Bigg|=O(\tau^3),\]
  \begin{eqnarray*}
    \mathbb{E}\Big(p^{k+1}-\overline{p}^{k+1}\Big)^2=\mathbb{E}\Big(\frac12\tau^2 AB\Big(J_k+\frac{1}{2\sqrt3}\Delta \mathcal{B}\Big)
  +O(\tau^3)\Big)^2=O(\tau^5),
  \end{eqnarray*}
  which leads to 
  \[\Bigg(\mathbb{E}\Big(p^{k+1}-\overline{p}^{k+1}\Big)^2\Bigg)^{1/2}=O(\tau^{5/2}).\]
  Following a similar approach, we also have 
  \[\Bigg|\mathbb{E}\Big(q^{k+1}-\overline{q}^{k+1}\Big)\Bigg|=O(\tau^3), \qquad
 \Bigg(\mathbb{E}\Big(q^{k+1}-\overline{q}^{k+1}\Big)^2\Bigg)^{1/2}=O(\tau^{5/2}).\]
Therefore, by applying the result of Lemma \ref{orderlemma}, the PRK method \eqref{prk} has mean-square order of $2$.
\end{proof}

\begin{remark}
The PRK scheme \eqref{prk} was shown in \cite{MRT2002} to have the mean-square order of 3/2 for general system. For this linear system \eqref{tq}, it can be shown to have second order mean-square convergence rate.
\end{remark}

\section{Numerical experiments}\label{numerical}
\setcounter{equation}{0} \setcounter{figure}{0}\setcounter{table}{0}

The numerical results of the DG scheme is presented in this section to demonstrate the performance of the proposed methods. 
We consider the DG method with various polynomial degree $k$ as the spatial discretization, and utilize the PRK method \eqref{prk} for temporal discretization.  The accuracy tests are provided for both 1D case and 2D case to demonstrate the convergence rate of the methods, and the linear growth of discrete energy is also studied for both cases.

\subsection{Accuracy test}
For the one-dimensional system \eqref{1} with periodic boundary conditions, we set $\lambda_1=\lambda_2=1$, and choose $W_t=\mathcal{B}_t$, which is the standard Brownian motion, so that one exact solution to \eqref{1} takes the form
  \begin{equation}\label{exact}
  \begin{cases}
\eta=\sin(x-t)+\cos(x+t)-\mathcal{B}_t&\\
u=\sin(x-t)-\cos(x+t)+\mathcal{B}_t.
\end{cases}
\end{equation} 

The computational domain is $[0,2\pi]$ and the final time is set to $T=3$. Initial conditions for $\eta(x,0)$ and $u(x,0)$ are 
obtained by letting $t=0$ in \eqref{exact}.  We use Nx and Nt to denote number of spatial steps and temporal steps respectively.  
Table \ref{t11} and Table \ref{t12} demonstrates the order of convergence rate for the case $k=1$ and $k=2$ respectively, where $\Delta t$ is chosen to be small enough to ensure that the spatial error dominates. Under both cases, we can observe the optimal convergence orders, i.e., $(k+1)$-th order of accuracy, which is consistent with the result in Theorem \ref{error estimate} for the semi-discrete method. Note that the second order temporal discretization is used for both cases, therefore, one would expect a second order convergence even coupled with third order DG spatial discretization. This second order convergence is observed in Table \ref{t13}, when larger $\Delta t$ is used.
 
\begin{table}[htb]
     \caption{Numerical error and convergence rates of 1D case when $k=1$.}
\centering
\begin{tabular}{|c|c|c|c|c|c|}
\hline 
Nx&Nt&$\norm{e_u}$&order&$\norm{e_\eta}$&order\\
\hline 
20&600&0.01295&0&0.01725&0\\
\hline
40&1200&3.201E-3&2.0162&4.281E-3&2.0106\\
\hline
80&2400&7.851E-4&2.0276&1.083E-3&1.9829\\
\hline
160&4800&1.958E-4&2.0036&2.714E-4&1.9966\\
\hline
\end{tabular}\label{t11}
\end{table}

\begin{table}[htb]
     \caption{Numerical error and convergence rates of 1D case when $k=2$.}
\centering
\begin{tabular}{|c|c|c|c|c|c|}
\hline 
Nx&Nt&$\norm{e_u}$&order&$\norm{e_\eta}$&order\\
\hline 
20&600&3.176E-4&0&4.407E-4&0\\
\hline
40&1200&3.958E-5&3.0045&5.486E-5&3.0061\\
\hline
80&2400&4.920E-6&3.0081&6.874E-6&2.9964\\
\hline
160&4800&6.144E-7&3.0013&8.645E-7&2.9912\\
\hline
\end{tabular}\label{t12}
\end{table}

   \begin{table}[htb]
     \caption{Numerical error and convergence rates of 1D case when $k=2$ with larger $\Delta t$.}
\centering
\begin{tabular}{|c|c|c|c|c|c|c|c|c|}
\hline 
Nx&Nt&$\norm{e_u}$&order&$\norm{e_\eta}$&order\\
\hline 
20&60&5.6019E-4&0&5.3547E-4&0\\
\hline
40&120&1.2175E-4&2.2020&9.5595E-5&2.4858\\
\hline
80&240&2.9228E-5&2.0585&2.0804E-5&2.2000\\
\hline
160&480&7.2337E-6&2.0145&4.9803E-6&2.0626\\
\hline
\end{tabular}\label{t13}
\end{table}

Next, we consider the two-dimensional stochastic Maxwell equations \eqref{2d} with periodic boundary conditions.
Set $\lambda_1=\lambda_2=1$, and choose $W_t=\mathcal{B}_t$, then the exact solution to \eqref{2d} takes the form
  \begin{equation}\label{2dexact}
  \begin{cases}
E=\sin(x+t)-\cos(y+t)-\mathcal{B}_t,\\
S=\cos(y+t)+\mathcal{B}_t,\\
T=\sin(x+t)+\mathcal{B}_t.
\end{cases}
\end{equation}
The spatial domain is set to be $[0,2\pi]^2$, and the final stopping time is taken to be $T=1$. The initial conditions of $E, ~S,~T$ can be obtained by setting $t=0$ in the exact solutions \eqref{2dexact}.  We use Nx and Nt to denote number of spatial steps and temporal steps respectively.  
The numerical errors and the corresponding convergence rates are shown in Table \ref{2dt1} for $k=1$ and in Table \ref{2dt2} for $k=2$. 
Under both cases, we can observe the optimal convergence orders, i.e., $(k+1)$-th order of accuracy, which matches 
the analysis in Theorem \ref{2derror} for the semi-discrete method. 

\begin{table}[htb]
     \caption{Numerical error and convergence rates of 2D case when $k=1$. }
\centering
\begin{tabular}{|c|c|c|c|c|c|c|c|c|}
\hline 
Nx&Ny&Nt&$\norm{e_E}$&order&$\norm{e_S}$&order&$\norm{e_T}$&order\\
\hline 
20&20&200&0.02715&0&0.01749&0&0.01749&0\\
\hline
40&40&400&6.617E-3&2.0366&4.466E-3&1.9696&4.466E-3&1.9696\\
\hline
80&80&800&1.634E-3&2.0176&1.129E-3&1.9839&1.129E-3&1.9839\\
\hline
160&160&1600&4.061E-4&2.0086&2.839E-4&1.9918&2.839E-4&1.9918\\
\hline
\end{tabular}\label{2dt1}
\end{table}

\begin{table}[htb]
     \caption{Numerical error and convergence rates of 2D case when $k=2$.}
\centering
\begin{tabular}{|c|c|c|c|c|c|c|c|c|}
\hline 
Nx&Ny&Nt&$\norm{e_E}$&order&$\norm{e_S}$&order&$\norm{e_T}$&order\\
\hline 
20&20&200&8.529E-4&0&5.899E-4&0&5.899E-4&0\\
\hline
40&40&400&9.980E-5&3.0952&7.061E-5&3.0628&7.061E-5&3.0628\\
\hline
80&80&800&1.253E-5&2.9939&9.119E-6&2.9529&9.119E-6&2.9529\\
\hline
160&160&1600&1.521E-6&3.0422&1.137E-6&3.0038&1.137E-6&3.0038\\
\hline
\end{tabular}\label{2dt2}
\end{table}

\subsection{Averaged energy growth} 
The discrete energy law satisfied by the numerical solutions $u_h$ and $\eta_h$ was studied in 
Theorem \ref{1denergy} for the one-dimensional system. Under the same setup as in the previous example, we simulate the
system up to the final stopping time $T=3$, with the time step size $\Delta t=0.0075$ and 80 space meshes. In this test, we run the simulations with $1000$ samples, and take the average to approximate the expectation and compute the averaged energy. Figure \ref{1denergyp} 
shows the time history of the averaged energy of our numerical solutions with different noise size $\lambda_{1,2}$, 
from which we can observe that the averaged energy is almost linear with respect to time. Note that when $\lambda_1=\lambda_2=0$,
 the global energy is preserved exactly on the discrete level. When the noise sizes $\lambda_{1,2}$ increase, the 
linear growth rate of the discrete energy also increases. The slopes of the lines in Figure \ref{1denergyp} (b-d) can be computed via least square fitting, and are approximately 0.1228, 3.0154, 12.4137, respectively. They are proportional to $\lambda_1^2+\lambda_2^2$, which is consistent with the result in Theorem \ref{1denergy}. 

 \begin{figure}[htb]
\centering
\subfigure[$\lambda_1=\lambda_2=0$]{\includegraphics[width=0.35\textwidth]{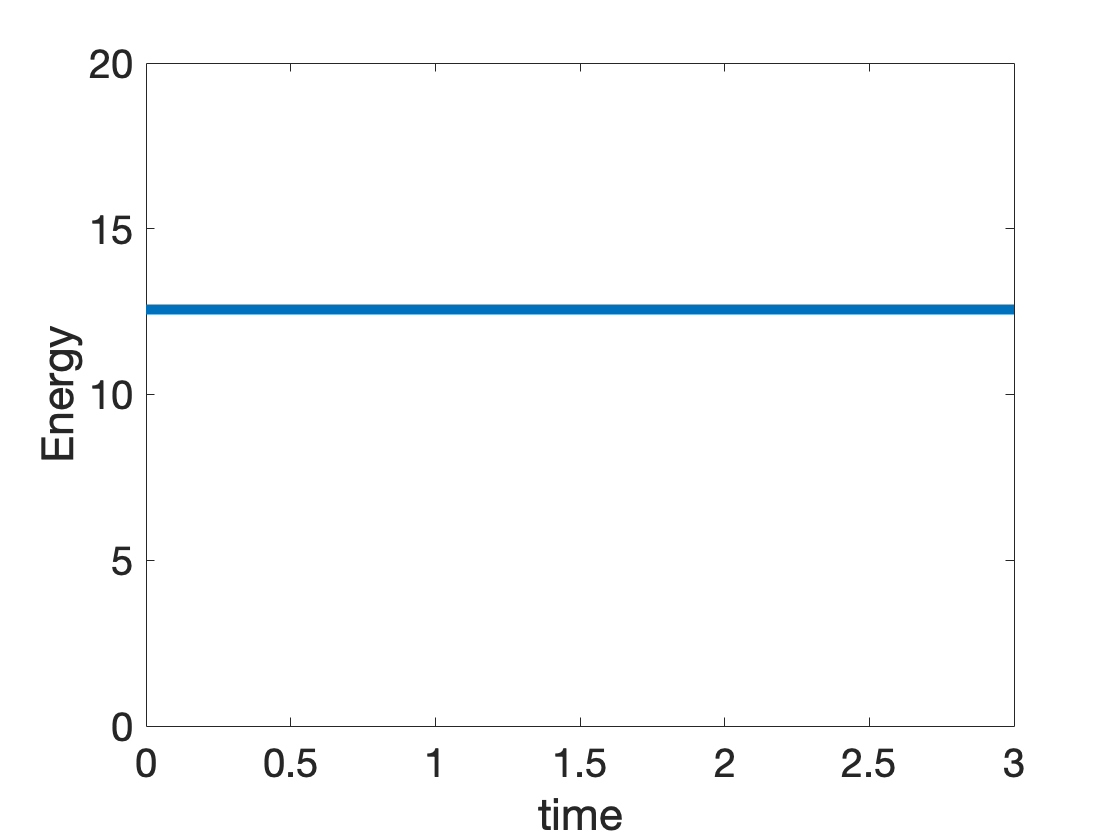}}
\subfigure[$\lambda_1=\lambda_2=0.1$]{\includegraphics[width=0.35\textwidth]{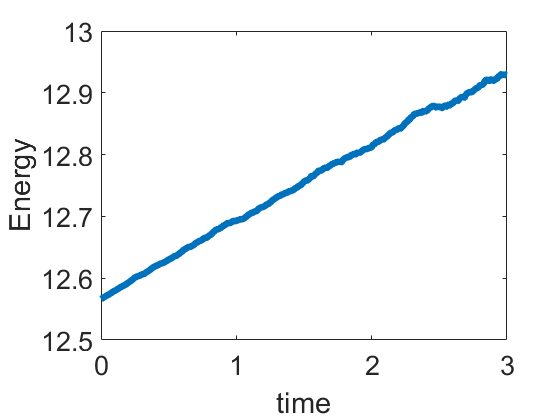}}
\subfigure[$\lambda_1=\lambda_2=0.5$]{\includegraphics[width=0.35\textwidth]{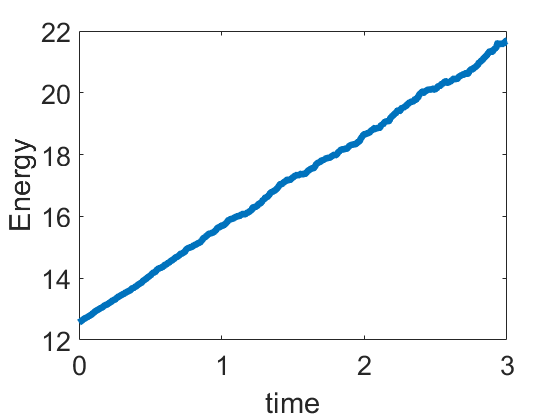}}
\subfigure[$\lambda_1=\lambda_2=1$]{\includegraphics[width=0.35\textwidth]{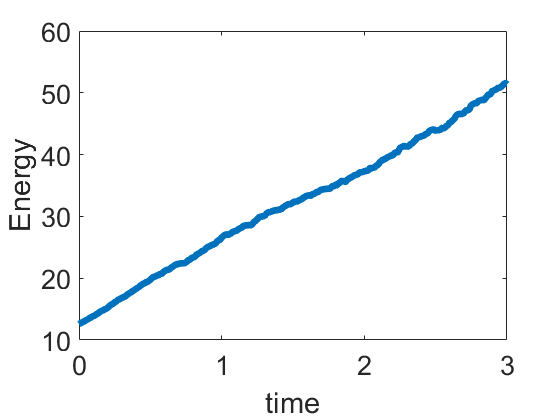}}
\caption{Averaged energy with different sizes of noise for 1D case. }
\label{1denergyp}
\end{figure}

Similarly, Theorem \ref{2denergy} studied the discrete energy law satisfied by the numerical solutions $E_h$, $S_h$ and $T_h$ for the two-dimensional system. 
 For this example we simulate the
system up to the final stopping time $T=1$, with the time step size $\Delta t=0.0025$ and $80\times 80$ space meshes. 
In this test, we run the simulations with $500$ samples, and take the average to approximate the expectation and compute the averaged energy. Figure \ref{2denergyp} 
shows the time history of the averaged energy of our numerical solutions with different noise size $\lambda_{1,2}$, 
from which we can observe that the averaged energy is almost linear with respect to time. Note that when $\lambda_1=\lambda_2=0$,
 the global energy is preserved exactly on the discrete level. When the noise sizes $\lambda_{1,2}$ increase, the 
linear growth rate of the discrete energy also increases. 
 The slopes of the lines in Figure \ref{2denergyp} (b-d) can be computed via least square fitting, and are approximately 1.2420, 29.9783, 117.2386, respectively. They are proportional to $\lambda_1^2+\lambda_2^2$, which is consistent with the result in Theorem \ref{2denergy}. 
 
 \begin{figure}[htb]
\centering
\subfigure[$\lambda_1=\lambda_2=0$]{\includegraphics[width=0.35\textwidth]{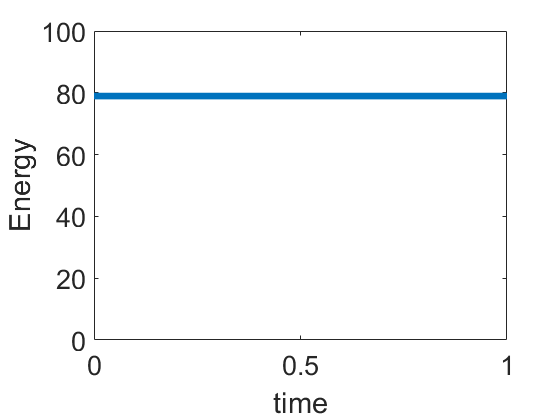}}
\subfigure[$\lambda_1=\lambda_2=0.1$]{\includegraphics[width=0.35\textwidth]{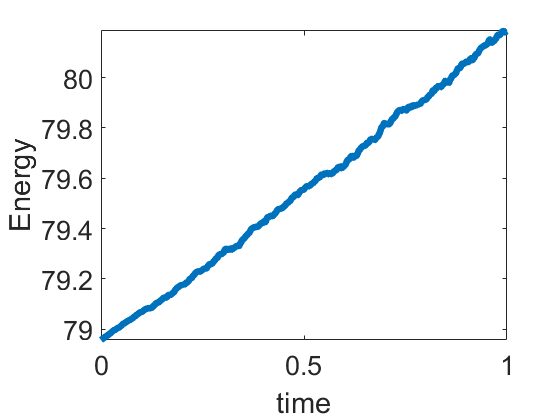}}
\subfigure[$\lambda_1=\lambda_2=0.5$]{\includegraphics[width=0.35\textwidth]{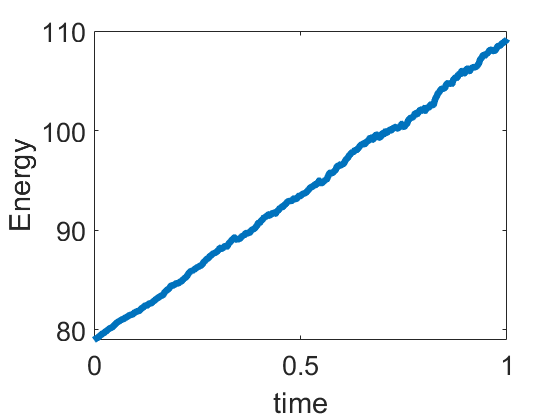}}
\subfigure[$\lambda_1=\lambda_2=1$]{\includegraphics[width=0.35\textwidth]{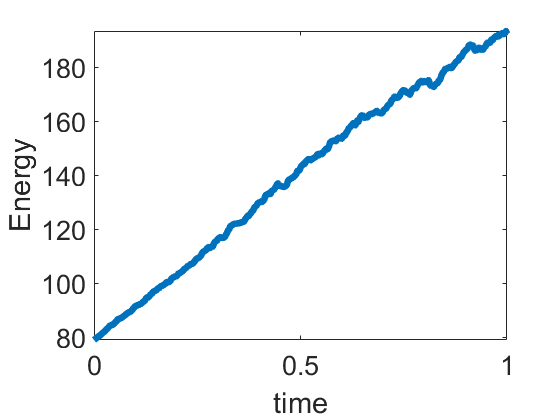}}
\caption{Averaged energy with different sizes of noise for 2D case.}
\label{2denergyp}
\end{figure}

\subsection{Test with noises of various sizes}\label{sec4.3}

In this example, we consider the two-dimensional system \eqref{2d}, and use the initial conditions studied in \cite{CHZ2016}:
\begin{eqnarray*}
  \begin{cases}
      E(x,y,0)=\sin(3\pi x)\sin(4\pi y),\\
  S(x,y,0)=-\frac45\cos(3\pi x)\cos(4\pi y),\\
  T(x,y,0)=-\frac35\sin(3\pi x)\cos(4\pi y),
  \end{cases}
\end{eqnarray*}
with $I\times J=[0,\frac23]\times[0,\frac12]$. The final stopping time is set as $T=1$. 
Following the definition \eqref{defW}, we construct the Wiener process as 
\begin{eqnarray}\label{constW}
  W_t=\sum^\infty_{m,n=1}2\sqrt{\frac{3}{m^3+n^3}}\sin\Big(\frac32m\pi x\Big)\sin\Big(2n\pi y\Big)\mathcal{B}_m(t),
\end{eqnarray}
and truncate the sum \eqref{constW} by taking the sum over $m,n$ from $1$ to $50$. 

In this example, the mesh sizes $\Delta x=\Delta y=0.0083$ and time step size $\Delta t=0.00083$ are considered. 
In order to show the effect of noise with various sizes on the numerical solution, we run the simulations with six sets of parameters: $\lambda_{1,2}=0.1, ~0.2, ~0.3, ~0.5, ~0.7,~1$. The contour plots of the numerical solution $S_h$ with these choices of noises are shown in Figure \ref{levelS}. The 3D plots of $T_h$ are also provided in Figure \ref{3dT}. It can be observed from these figures that the perturbation of the numerical solutions becomes more and more obvious as the size of the noise grows. 

\begin{figure}[htb]
\centering
\subfigure[$\lambda_1=\lambda_2=0.1$]{\includegraphics[width=0.45\textwidth]{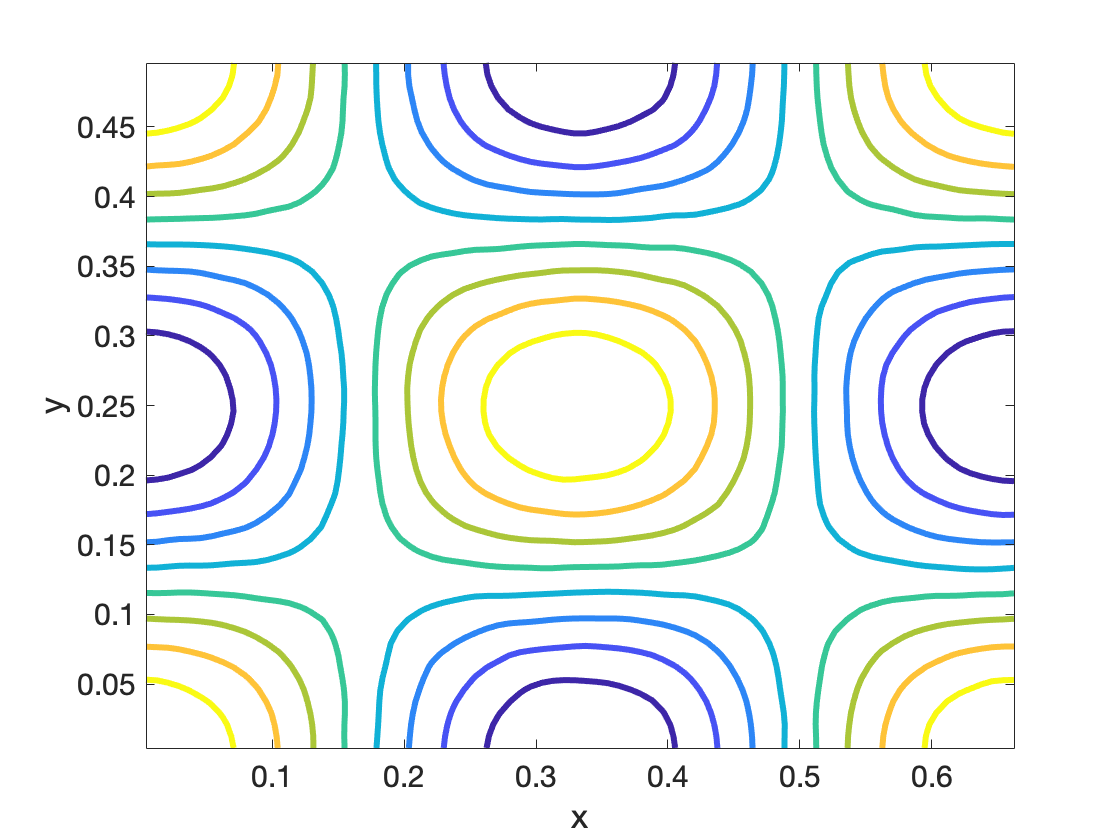}}
\hspace{-1cm}
\subfigure[$\lambda_1=\lambda_2=0.2$]{\includegraphics[width=0.45\textwidth]{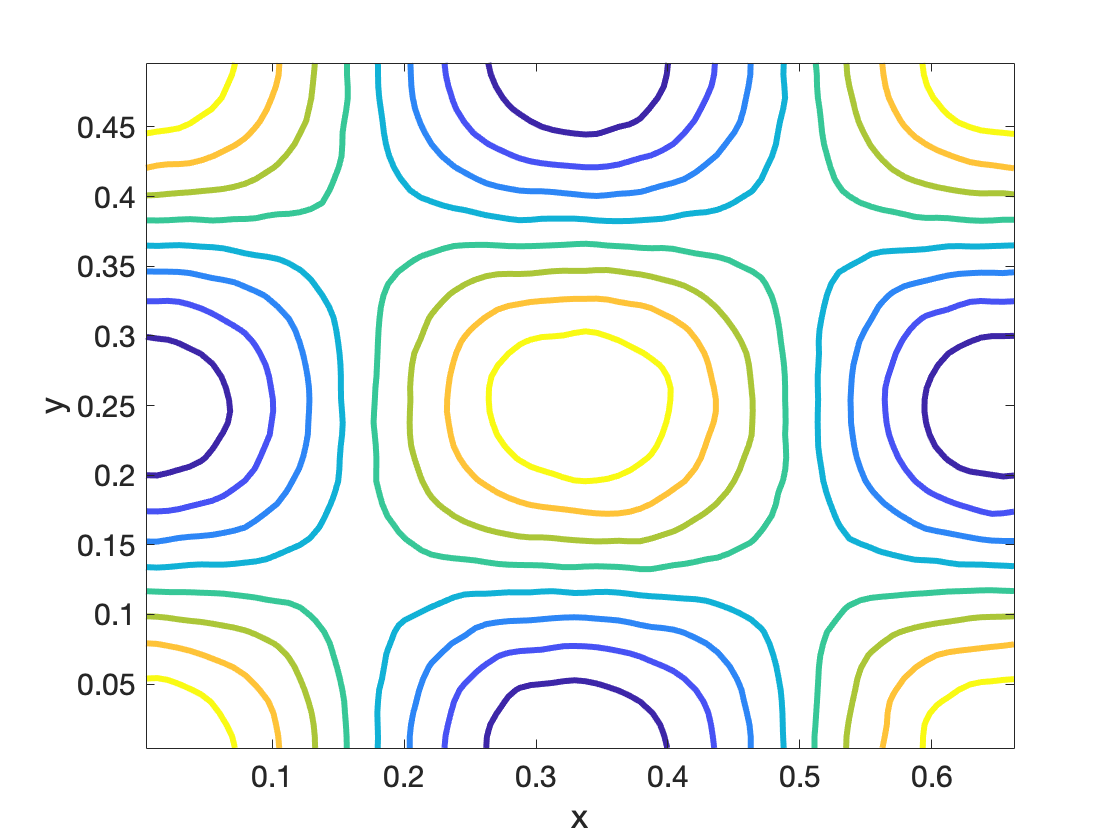}} \\
\subfigure[$\lambda_1=\lambda_2=0.3$]{\includegraphics[width=0.45\textwidth]{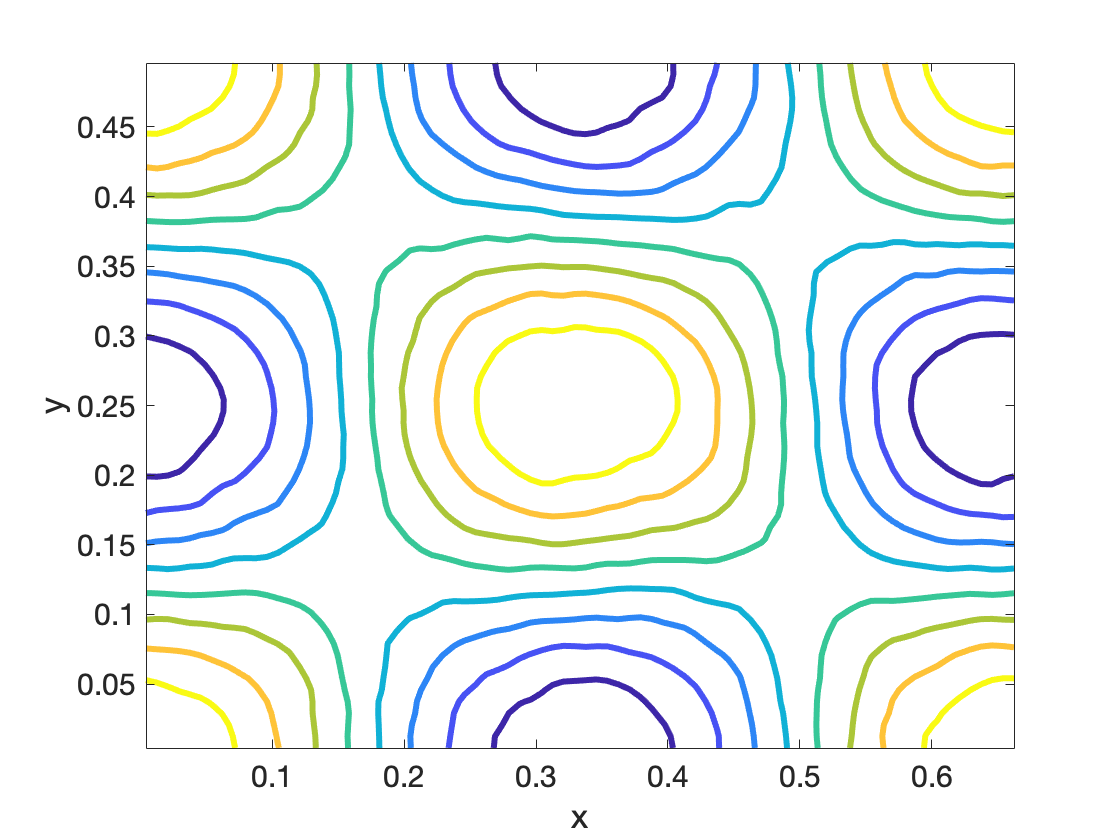}}
\hspace{-1cm}
\subfigure[$\lambda_1=\lambda_2=0.5$]{\includegraphics[width=0.45\textwidth]{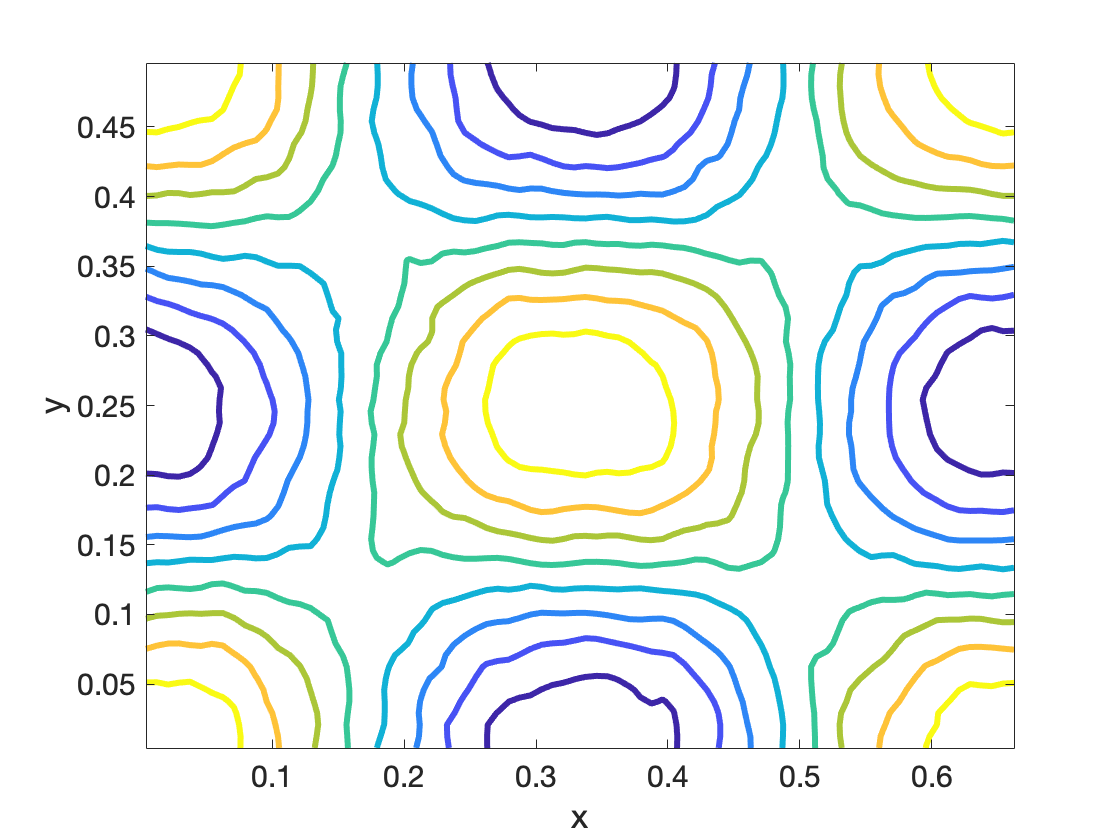}} \\
\subfigure[$\lambda_1=\lambda_2=0.7$]{\includegraphics[width=0.45\textwidth]{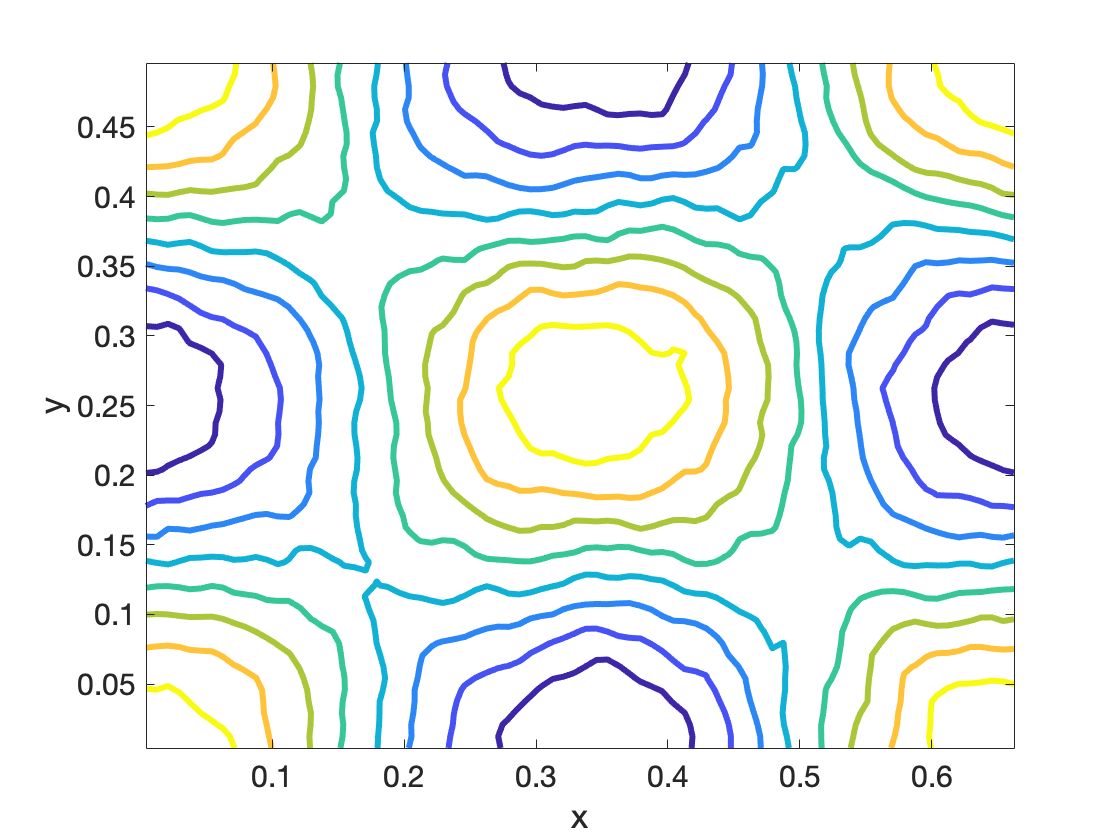}}
\hspace{-1cm}
\subfigure[$\lambda_1=\lambda_2=1$]{\includegraphics[width=0.45\textwidth]{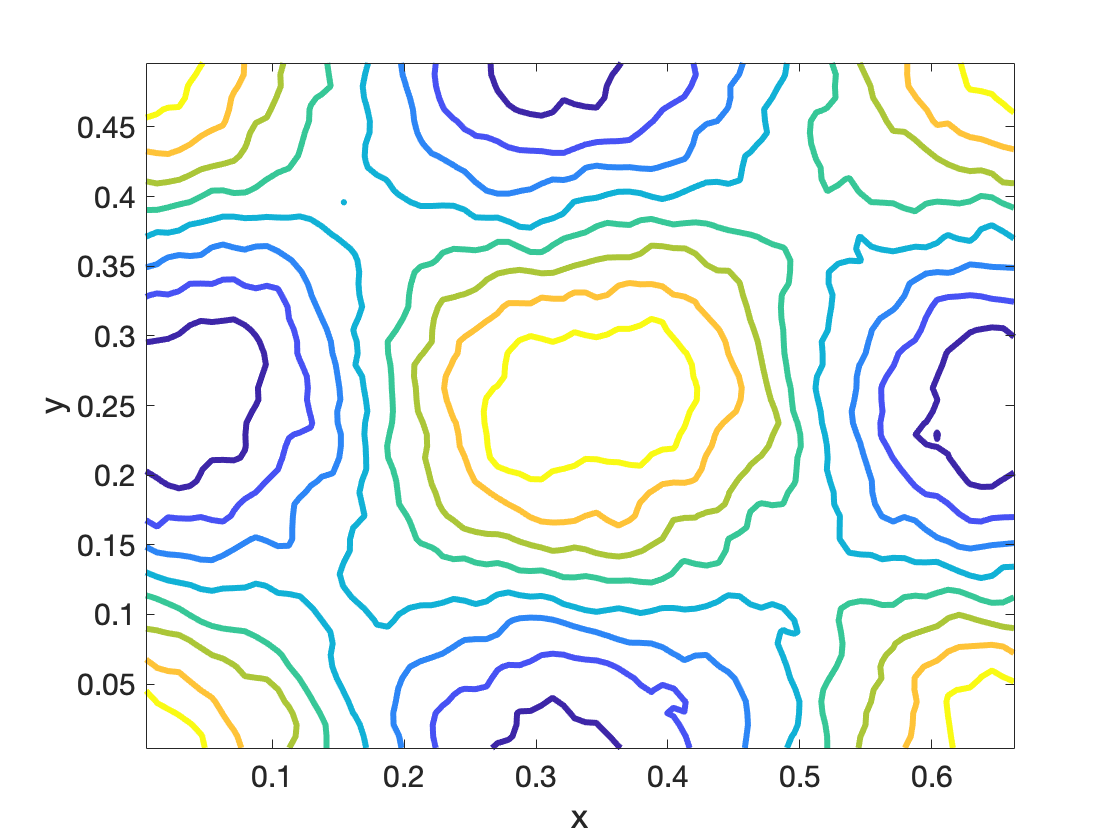}}
\caption{Contour plots of $S_h(x,y)$ with different sizes of noise for the test in Section \ref{sec4.3}. $8$ uniformly spaced contour lines  within the range $[-0.7, 0.7]$ are used. }
 \label{levelS}
\end{figure}

\begin{figure}[htb]
\centering
\subfigure[$\lambda_1=\lambda_2=0.2$]{\includegraphics[width=0.45\textwidth]{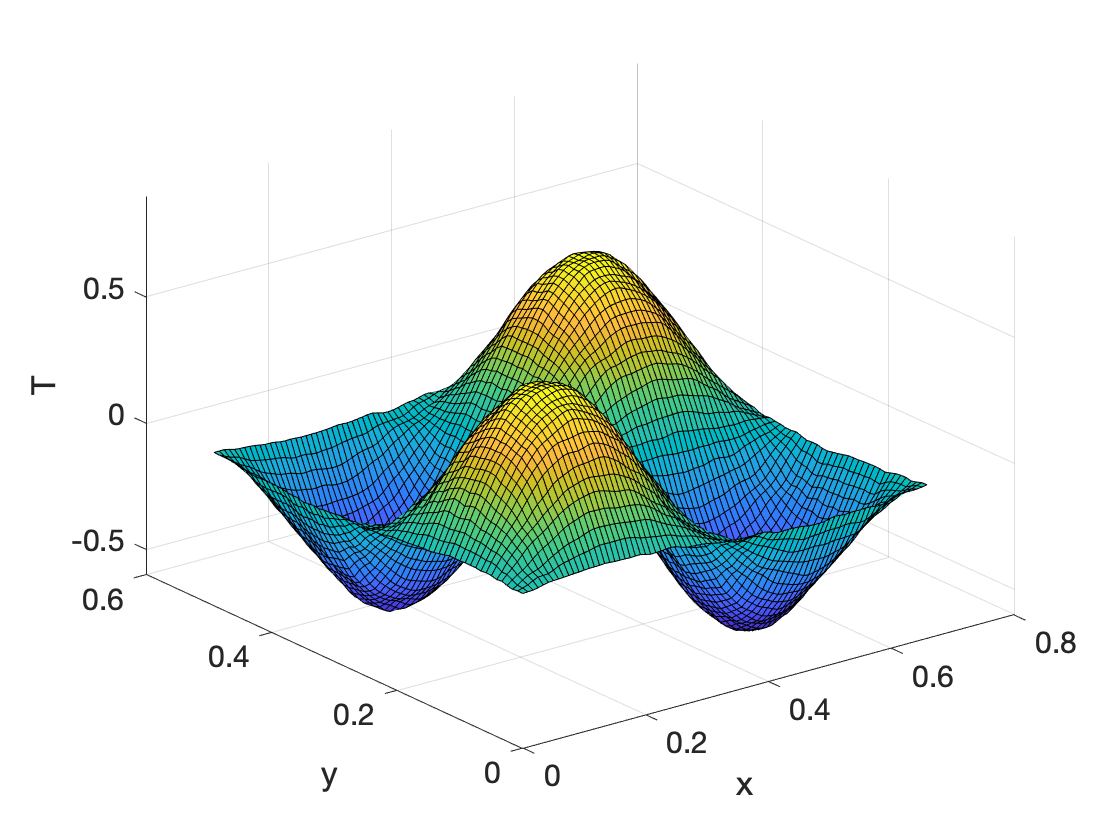}}
\subfigure[$\lambda_1=\lambda_2=0.5$]{\includegraphics[width=0.45\textwidth]{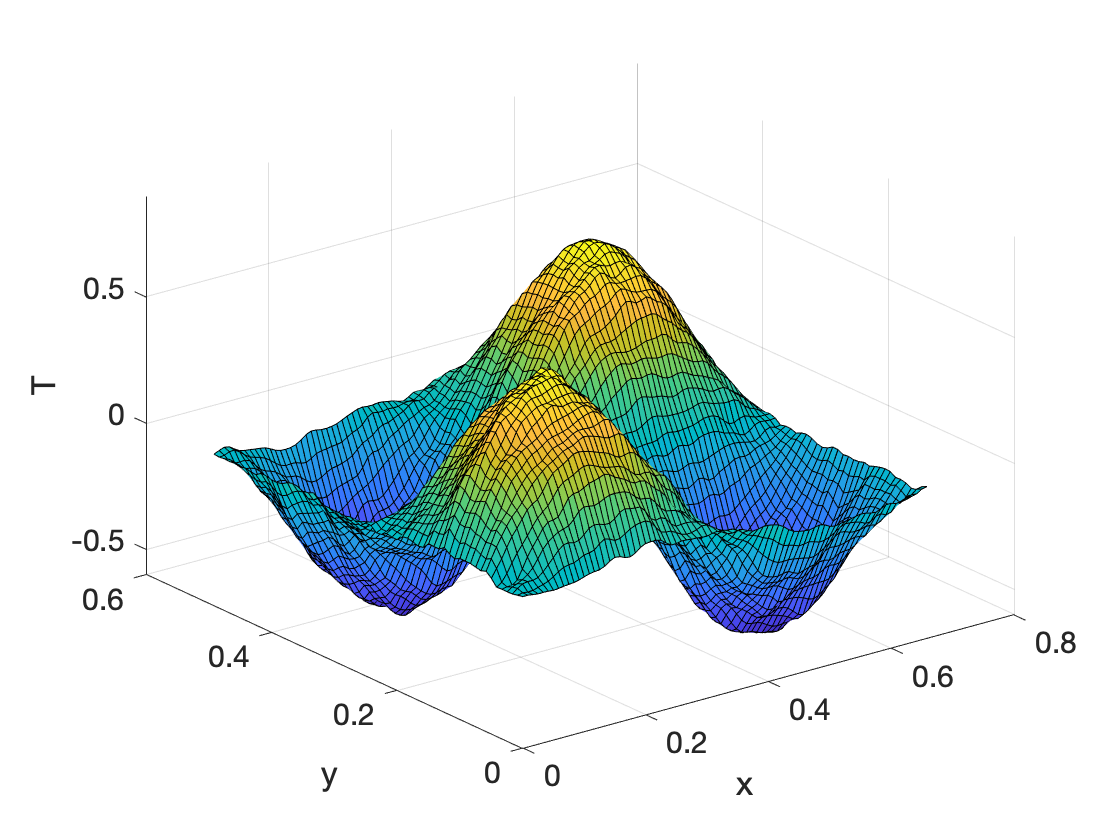}}\\
\subfigure[$\lambda_1=\lambda_2=0.7$]{\includegraphics[width=0.45\textwidth]{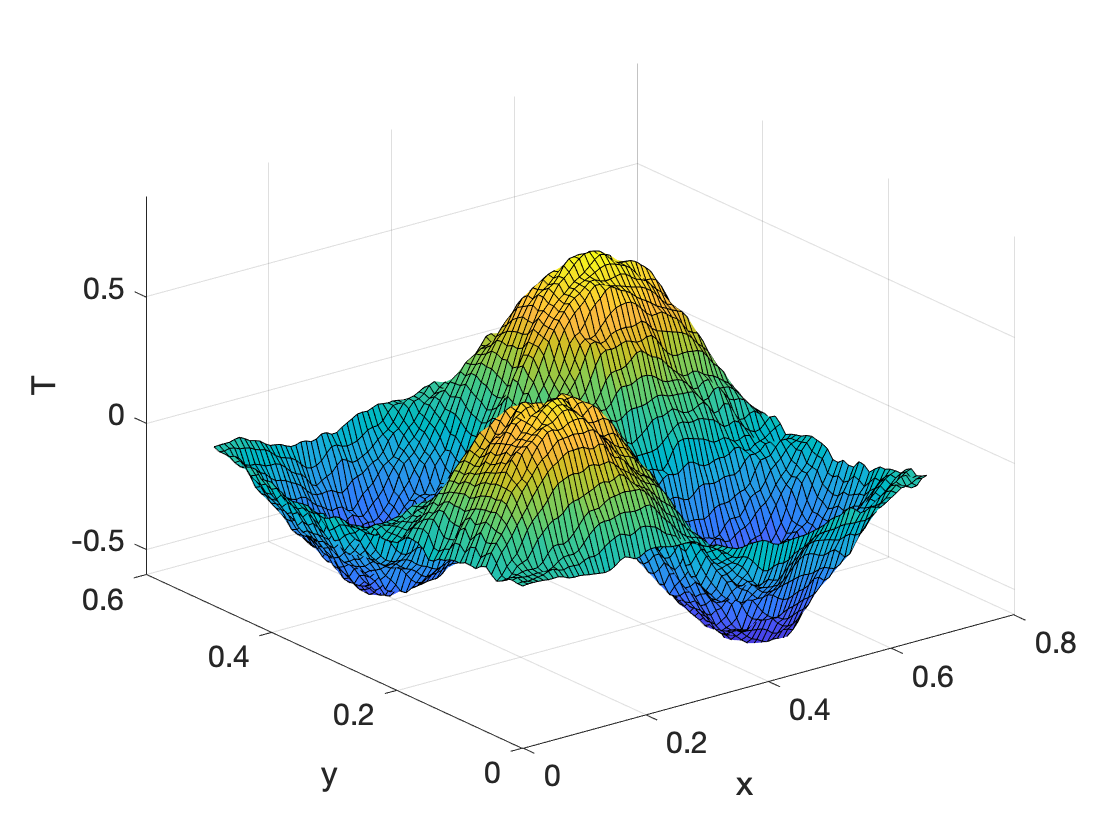}}
\subfigure[$\lambda_1=\lambda_2=1$]{\includegraphics[width=0.45\textwidth]{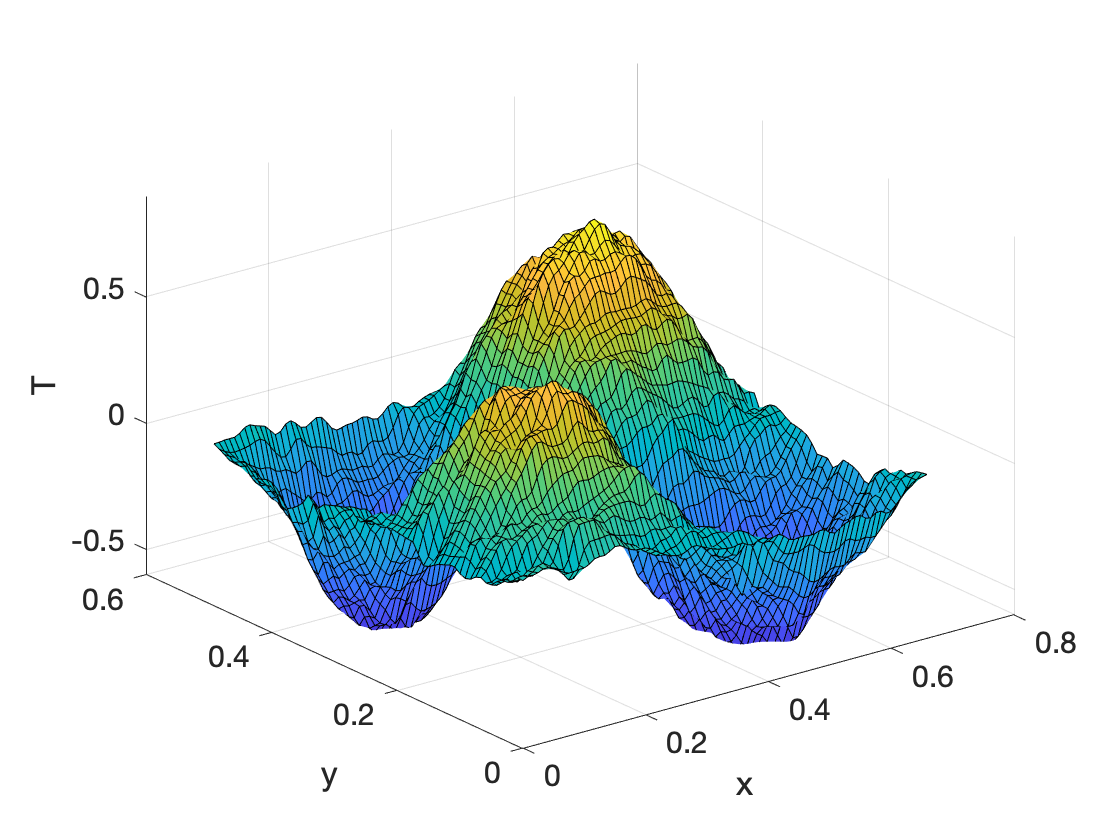}}
\caption{3D plots of $T_h(x,y)$ with different sizes of noise for the test in Section \ref{sec4.3}. } 
 \label{3dT}
\end{figure}

\section{Conclusion remark}\label{conclusion}
\setcounter{equation}{0} \setcounter{figure}{0}\setcounter{table}{0}

In this paper we have developed and analyzed the DG scheme for the one- and two-dimensional 
stochastic Maxwell equations with additive noise. The proposed methods are shown to satisfy the discrete form of the stochastic 
energy linear growth property. The optimal error estimate of the semi-discrete methods is also proven analytically. 
By introducing auxiliary variables, we also rewrite the stochastic Maxwell equations into the multi-symplectic structure and demonstrate that the proposed DG methods preserve the multi-symplectic structure. The semi-discrete methods are then combined with the symplectic Euler or PRK temporal discretization methods. Numerical experiments are provided to test the performance of the resulting methods, and optimal error estimates and linear growth of the discrete energy can be observed for all cases.

\end{document}